\documentclass[reqno]{amsart}
\usepackage[scale=0.75, centering, headheight=14pt]{geometry}
\usepackage[latin1]{inputenc}
\usepackage[T1]{fontenc}
\usepackage{lmodern}
\usepackage[english]{babel}
\usepackage{microtype}
\usepackage[fontsize=11pt]{scrextend}
\usepackage{amsmath,amssymb,amsfonts,amsthm}
\usepackage{mathtools,accents}
\usepackage{mathrsfs}
\usepackage{aliascnt}
\usepackage{braket}
\usepackage{bm}
\usepackage[initials]{amsrefs}
\usepackage[citecolor=blue,colorlinks]{hyperref}
\addto\extrasenglish{}

\usepackage{enumerate}
\usepackage{xcolor}

\usepackage{aliascnt}

\makeatletter
\def\newaliasedtheorem#1[#2]#3{
  \newaliascnt{#1@alt}{#2}
  \newtheorem{#1}[#1@alt]{#3}
  \expandafter\newcommand\csname #1@altname\endcsname{#3}
}
\makeatother

\theoremstyle{plain}
\newtheorem{theorem}{Theorem}[section]
\newaliasedtheorem{lemma}[theorem]{Lemma}
\newaliasedtheorem{prop}[theorem]{Proposition}
\newaliasedtheorem{claim}[theorem]{Claim}
\newaliasedtheorem{corollary}[theorem]{Corollary}

\theoremstyle{definition}
\newaliasedtheorem{definition}[theorem]{Definition}
\newaliasedtheorem{example}[theorem]{Example}

\theoremstyle{remark}
\newaliasedtheorem{remark}[theorem]{Remark}

\numberwithin{equation}{section}

\def\R{\mathbb R}
\title[Dynamics for Dipolar BEC]{Dynamics of solutions to the Gross-Pitaevskii equation describing dipolar Bose-Einstein condensates} 

\author[J. Bellazzini \and L. Forcella]{Jacopo Bellazzini \and Luigi Forcella}

\address[J. Bellazzini]{Dipartimento di Matematica, Universit\`a Degli Studi di Pisa, Largo Bruno Pontecorvo, 5, 56127, Pisa, Italy}
\email{jacopo.bellazzini@unipi.it}

\address[L. Forcella]{Department of Mathematics, Heriot-Watt University, and The Maxwell Institute for the Mathematical Sciences, Edinburgh, EH14 4AS, United Kingdom}
\email{l.forcella@hw.ac.uk}



\begin{document}

\maketitle

%\tableofcontents{}

\begin{abstract}
We review some recent results on the long time dynamics of solutions to the Gross-Pitaevskii equation (GPE) governing non-trapped dipolar Quantum Gases. We describe the asymptotic behaviours of solutions for different initial configurations of the initial datum in the energy space, specifically for data below, above, and at the Mass-Energy threshold.  We revisit some properties of powers of the Riesz transforms by means of the decay properties of the heat kernel associated to the parabolic biharmonic equation. These decay properties play a fundamental tool in establishing  the dynamical features of the solutions to the studied GPE. 
\end{abstract}

\section{Introduction} 
In this paper, we review some recent progresses concerning the dynamics of solutions to the following Gross-Pitaevskii equation (GPE) which models a so-called dipolar Bose-Einstein Condensate (BEC) at low temperatures, see \cite{AEMWC, BrSaToHu,  NaPeSa, BaCa, DMAVDKK, SSZL, LMS,PS,YY1,YY2}:
\begin{equation}
\label{eq:evolution}
i h \frac{\partial u}{\partial t} = - \frac{h^2}{2m}\Delta u + W(x) u + U_0|u|^2 u + (V_{dip}\ast |u|^2) u.
\end{equation}
In the equation above, $t$ is the time variable, $x = (x_1,x_2,x_3) \in \R^3 $ is the space variable, $\ast$ denotes the convolution, and $u=u(t,x)$ is a complex function.  The physical parameters appearing in \eqref{eq:evolution} are: the Planck constant $h$,  the mass $m$ of a dipolar particle, $U_0 = 4 \pi h^2 a_s /m$ describes the strength of the local interaction between dipoles in the condensate,  where  $a_s$ the $s-$wave scattering length, which may have positive or negative sign according to the repulsive/attractive nature of the interaction. The non-local, long-range dipolar interaction potential between two dipoles is given instead by the convolution through the potential 
\begin{equation*}
\label{eq:dipole}
V_{dip}(x) = \frac{\mu_0 \mu^2_{dip}}{4 \pi} \,  \frac{1 - 3 \cos^2 (\theta)}{|x|^3}, \quad x \in \R^3,
\end{equation*}
where $\mu_0$ is the vacuum magnetic permeability, $\mu_{dip}$ is the permanent magnetic dipole moment, and $\theta$ is the angle between the dipole axis and the vector $x$.  Without loss of generality, we can assume the dipole axis to be the vector $(0,0,1)$.   The potential $W(x)$ is an external trapping potential which will be not considered in the sequel, namely we study the case $W(x)=0$.\\ 

For a mathematical treatment of the equation above, we consider \eqref{eq:evolution} in its dimensionless form, and in particular we study the associated Cauchy problem in the energy space (i.e. $H^1(\R^3)$) as follows:
\begin{equation}\label{GP}
\left\{ \begin{aligned}
i\partial_{t}u+\frac12\Delta u&=\lambda_1|u|^{2}u+\lambda_2(K\ast|u|^2)u, \quad (t,x)\in \R\times 
\mathbb{R}^3\\
u(0,x)&=u_0(x) \in H^1(\mathbb R^3)
\end{aligned}\right.,
\end{equation}
where  the dipolar kernel $K$ is now given by 
\begin{equation}\label{eq:ker}
K(x)=\frac{x_1^2+x_2^2-2x_3^2}{|x|^5}.
\end{equation}
Provided we normalize the the wave function according to $\int_{\R^3} |u(x,t)|^2 dx = N$, whereas $N$ is the total number of dipolar particles in the dipolar BEC, then the two real coefficients  $\lambda_1$ and $\lambda_2$ are defined by 
$\lambda_1 = 4 \pi a_s N \sqrt{\frac mh}$, and $ \lambda_2 = \frac{N \mu_0 \mu_{dip}^2 }{4 \pi }\sqrt{\frac{m^3}{h^5}}$, and they are two physical parameters describing the strength of the nonlinearities involved in the equation, specifically the local one given by $|u|^2u$, and the non-local one given by $(K\ast |u|^2)u$, respectively.

At least formally, the solution $u(t)$ to \eqref{GP} preserve the  Mass and the Energy of the initial datum $u(0)=u_0$, specifically 
	\begin{equation}\label{eq:mass}
	M(u(t)) := \|u(t)\|^2_{L^2(\R^3)} = M(u(0)), 
	\end{equation}
	and
	\begin{equation}\label{eq:energy}
	E(u(t)) := \frac{1}{2}\int_{\R^3}|\nabla u(t)|^2 dx+\frac12  \int_{\R^3} \lambda_1 |u(t)|^4 + \lambda_2 (K\ast |u(t)|^2)|u(t)|^2 dx 
= E(u(0)),
	\end{equation}
	where 	$M(u(t)) $ and $E(u(t)) $ define the Mass and the Energy, respectively.
For later purpose, we introduce the notation  	
	\begin{equation*}\label{defi-H}
	H(f):= \|\nabla f\|^2_{L^2(\R^3)}
	\end{equation*}
	for the kinetic energy, and 
	\begin{equation*}\label{defi-N}
	P(f):= \int_{\R^3} \lambda_1 |f(x)|^4 + \lambda_2 (K\ast |f(x)|^2)|f(x)|^2 dx 
	\end{equation*}
for the potential energy; hence we rewrite  	
\[
E(u(t))=\frac{1}{2}\left(H(u(t))+P(u(t))\right).
\]	
Assuming a local-in-time existence theory for \eqref{GP} (which is guaranteed by the work of Carles, Markowich, and Sparber, see \cite{CMS}), and by assuming enough regularity of the solutions, the conservation  laws  \eqref{eq:mass} and \eqref{eq:energy} can be proved by a simple integration by parts; a rigorous justification in the energy space $H^1(\R^3)$ (note that in this Sobolev space the energy functional  is well-defined) can be done by an approximation argument. Besides the  functionals $E$, $H$, and $P$ above, we introduce the Pohozaev functional
\begin{equation}\label{def:G}
G(f):= H(f) +\frac{3}{2} P(f).
\end{equation}
It is worth observing that the functional $G$ is (up to a $1/4$ factor) the second derivative in time of the virial functional associated to \eqref{GP}, i.e. 
\[
G(u(t))= \frac14\frac{d^2}{dt^2} V(t),
\]
where $V(t):=V(u(t))$ stands for the variance at time $t$ of the mass density, namely 
\begin{align} \label{defi-V}
	V(t):= \int_{\R^3}|x|^2|u(t,x)|^2 dx.
\end{align}
Motivated by the definition of the functional $V$, we introduce the space of functions $\Sigma\subset H^1(\R^3)$ as  $\Sigma:= H^1(\R^3) \cap L^2(\R^3;|x|^2 dx).$ \medskip

Following the  work by Carles, Markowich, and Sparber \cite{CMS},   we  introduce the partition of the coordinate plane $(\lambda_1,\lambda_2)$ given by the two sets below:
\begin{equation}\label{UR}
UR:=\left\{ \begin{aligned}
\lambda_1-\frac{4\pi}{3}\lambda_2<0 &\quad \hbox{ if } \quad \lambda_2>0\\
\lambda_1+\frac{8\pi}{3}\lambda_2<0 & \quad\hbox{ if } \quad\lambda_2<0
\end{aligned}\right.,
\end{equation}
and its complementary set in $\mathbb R^2$, namely
\begin{equation}\label{SR}
SR:=\left\{ \begin{aligned}
\lambda_1-\frac{4\pi}{3}\lambda_2\geq 0 &\quad \hbox{ if } \quad \lambda_2>0\\
\lambda_1+\frac{8\pi}{3}\lambda_2\geq 0 & \quad\hbox{ if } \quad\lambda_2<0
\end{aligned}\right..
\end{equation}
The two sets above are called  \emph{Unstable Regime} (see \eqref{UR}) and   \emph{Stable Regime} (see \eqref{SR}), respectively. \\

The separation of the parameters $\lambda_1$ and $\lambda_2$ as in the regions \eqref{UR} and  \eqref{SR}, is crucial in establishing the dynamics of solutions to \eqref{GP}. Indeed,  there are two main differences when working in the Unstable Regime instead of  the Stable regime. Firstly, in \eqref{UR} the conservation of the energy does not imply a boundedness in the kinetic term, second the  solutions to the  stationary equation (see \eqref{ell-equ} below) associated to \eqref{GP} do exist. Hence, at least in a naive way, we can think to the Unstable/Stable Regimes as the analogous for the Gross-Pitaevskii equation \eqref{GP} of the focusing/defocusing characters for the usual cubic NLS equation. However, note that here it is improper to speak about defocusing/focusing character for \eqref{GP}, since even for two positive coefficients of the non-linear terms $0<\lambda_1<\frac{4\pi}{3}\lambda_2$ finite time blow-up solutions may come up. See \cite[Lemma 5.1]{CMS}, where negative energy solutions are constructed. We also mention here that in the Stable Regime, we proved in \cite{BF19}, that for any initial datum in $H^1(\R^3)$ the corresponding solution to \eqref{GP} is global in time and scatters.\\

Similarly to the classical NLS equation (and more in general to other dispersive PDEs), a fundamental tool towards a classification of Cauchy data $u_0\in H^1(\mathbb R^3)$ as in \eqref{GP} leading to global (and scattering) solutions versus blowing-up solutions, is given by means of quantities related to the solutions of the stationary equation associated to \eqref{GP}:
\begin{equation} \label{ell-equ}
-\frac{1}{2} \Delta Q_{\mu} + \mu Q_{\mu} + \lambda_1 |Q_{\mu}|^2 Q_{\mu} +\lambda_2 (K\ast |Q_{\mu}|^2) Q_{\mu}=0,\quad \mu>0.
\end{equation}
Notice that if  $Q_{\mu}$ solves \eqref{ell-equ} then $u(t,x):=e^{-i \mu t}Q_{\mu}(x)$ solves  \eqref{GP}. Moreover, by an elementary scaling argument,  $E(Q_\mu) M(Q_{\mu})=E(Q_1) M(Q_1)$ for all $\mu>0$. For sake of simplicity in the notation, we will call $Q$ the standing wave solutions with $\mu=1$.
In particular, some bounds for the product of the Mass and the Energy of an initial datum in terms of the Mass and Energy of solutions  $Q$ to \eqref{ell-equ} allow to determine wether a solution $u(t)$ to \eqref{GP} exists for all time and scatters, or formation of singularities in finite (or infinite) time may arise. Indeed, sufficient conditions on $u_0\in H^1(\R^3)$ for the scattering/blow-up scenario are given by the relations below:
	\begin{equation}\label{sc-reg}
	(SC):=\left\{
		\begin{aligned} 
		E(u_0)M(u_0) &< E(Q) M(Q) \\
		H(u_0) M(u_0)&< H(Q) M(Q)
		\end{aligned}
	\right.,
	\end{equation}
and 
	\begin{equation}\label{blow-reg}
	(BC):=\left\{
		\begin{aligned} 
		E(u_0)M(u_0) &< E(Q) M(Q) \\
		H(u_0) M(u_0)&> H(Q) M(Q)
		\end{aligned}
	\right.,
	\end{equation}
respectively.
The above conditions on initial data are referred as the Mass-Energy (of the initial datum) below the threshold, the latter given by the quantity  $E(Q) M(Q)$. \\

As mentioned above, in the Unstable Regime \eqref{UR}, existence of solutions to \eqref{ell-equ} do exist, and it was proved in two different papers by Antonelli and Sparber, see \cite{AS}, and later by the first author and Jeanjean, see \cite{BJ}, by employing two different methods. In the former work, existence of ground states (i.e. standing wave solutions that minimize the energy functional $E(u)$ among all the standing solutions  with prescribed mass) are proved by means of minimizing a Weinstein-type functional, while in the latter a geometrical approach is used, specifically by proving that the energy functional satisfied a mountain pass geometry.  As for the usual cubic NLS, it turns out that a ground state $Q$ related to the elliptic equation gives an optimizer for the Gagliardo-Nirenberg-type inequality 
\begin{align} \label{GN-ineq}
-P(f) \leq C_{GN} (H(f))^{\frac{3}{2}} (M(f))^{\frac{1}{2}},
\end{align}
for $ f \in H^1(\R^3)$, meaning that $C_{GN} = -P(u_0)/ (H(Q))^{\frac{3}{2}} (M(Q))^{\frac{1}{2}}$.  Furthermore, the Pohozaev identities tell us that $H(Q) = 6 M(Q)=-\frac{3}{2} P(Q)$, and by the latter relations we have that $E(Q) = \frac{1}{6}H(Q) = -\frac{1}{4}P(Q)$ and that  
	\begin{align} \label{inde-quant-proof}
	E(Q) M(Q) = \frac{1}{6} H(Q) M(Q) = -\frac{1}{4} P(Q) M(Q) = \frac{2}{27} (C_{GN})^{-2}.
	\end{align}
It is important to remark that uniqueness of ground states -- even up to the action of some symmetry -- is unknown; nonetheless, by \eqref{inde-quant-proof} we can see that the quantities $E(Q) M(Q)$, $H(Q) M(Q) $, and $P(Q) M(Q) $ are independent of the choice  of the ground state.\\

In the paper, we will also give dynamics results for solutions with arbitrarily large initial data (although by imposing some other hypothesis on $u_0$ and/or by further restricting the conditions on the parameters $\lambda_1$ and $\lambda_2$ to a subset of the Unstable Regime), hence by considering data such that $E(u_0)M(u_0) >E(Q) M(Q)$, and for data exactly at the threshold, i.e. for data satisfying $E(u_0)M(u_0) = E(Q) M(Q)$. See the next subsection, where we enunciate the main results on the dynamics of solutions to \eqref{GP}.

\subsection{Main results} We conclude the Introduction by stating the main results contained in the paper. We separate them according the fact that the initial data are below, above, or at the threshold determined by $E(Q) M(Q)$.
\subsubsection{Dynamics below the threshold}  We start by giving the scattering Theorem and the blow-up in finite time Theorem, for solutions to \eqref{GP} arising from initial data below the Mass-Energy threshold, described in terms of a solution $Q$ of the elliptic equation \eqref{ell-equ}.   In what follows, $ e^{it\frac{1}{2}\Delta}$ denotes the linear Schr\"odinger  propagator, namely $v(t,x)=e^{it\frac{1}{2}\Delta}v_0$ solves $i\partial_t v+\frac12\Delta v=0$, with $v(0,x)=v_0$.
As already mentioned above, local well-posedness for \eqref{GP} was established in \cite{CMS}, by a usual fixed point argument based on Strichartz spaces, and upon having established some basic properties on the convolution kernel $K$, see \autoref{prop:cont-K} and \autoref{lemma:fou-k} below.  In what follows, with denote by $T_{min}>0$ and $T_{max}>0$ the minimal and the maximal times of existence of a solution to \eqref{GP}.  

The asymptotic dynamics for data below the threshold has been proved by the authors in \cite{BF19} and \cite{BF21}. In \cite{BF19}, we proved the following.
\begin{theorem}\label{theo-scat-BF}
		Let $\lambda_1$ and $\lambda_2$ satisfy \eqref{UR}, namely they belong to the Unstable Regime. Let $u_0 \in H^1(\R^3)$ satisfy \eqref{sc-reg}, where $Q$ is a ground state related to \eqref{ell-equ}. Then the corresponding solution $u(t)$ to \eqref{GP} exists globally in time and scatters in $H^1(\R^3)$ in both directions, that is, there exist $u^\pm_0 \in H^1(\R^3)$ such that
\begin{equation*} \label{defi-scat}
\lim_{t\rightarrow \pm \infty} \|u(t)- e^{it\frac{1}{2}\Delta} u^\pm_0\|_{H^1(\R^3)} =0.
\end{equation*}
\end{theorem}

\noindent The theorem above is given by implementing a concentration/compactness  and rigidity scheme, as we will explain in the next subsections.	\\

In order to give the blow-up results that we proved in \cite{BF21}, let us define by $\bar x=(x_1 ,x_2 )$, and let us introduce the functional space where the occurrence of formation of singularities in finite time is established: 
\begin{equation*}
\Sigma_3 =\left\{ u \in H^1(\R^3)\quad s.t. \quad u(x)=u(|\bar x|, x_3) \ \hbox{ and } \ u\in L^2(\R^3;x_3^2\,dx) \right\}.
\end{equation*}
$\Sigma_3$ is therefore the space of cylindrical symmetric functions with finite variance in the $x_3$ direction. We have the following.

\begin{theorem}\label{thm:main}
Assume that $\lambda_1$ and $ \lambda_2$ satisfy \eqref{UR}, namely they belong to the Unstable Regime. Let $u(t)\in \Sigma_3$ be a solution to \eqref{GP} defined on $(-T_{min}, T_{max}),$ with initial datum $u_0$ satisfying  \eqref{sc-reg}, where $Q$ is a ground state related to \eqref{ell-equ}. Then $T_{min}$ and $T_{max}$ are finite, namely $u(t)$ blows-up in finite time. 
\end{theorem}

It is worth mentioning that both for the scattering and the blow-up result, the main difficulty with respect to other NLS non-local models, is the precise structure of the dipolar kernel. Moreover, no radial symmetry for the solutions can be assumed in our context, as the convolution with radial  function  would make disappear the contribution of the non-local term, hence reducing the equation to a standard cubic NLS. Thus, the blow-up result above for cylindrical symmetric solution is somehow the best one may obtain; let us recall that finite time blow-up without assuming any structure on the solutions is still unknown even for the usual focusing cubic NLS equation. Moreover, we point-out that the dipolar kernel $K$ enjoys a cylindrical symmetry, so our assumption is also physically consistent.\\

As said above, similarly to the classical cubic focusing NLS, if we do not assume any additional hypothesis on the initial datum, as in \autoref{thm:main} for example, we cannot prove that the solutions blow-up in finite time. Nonetheless, in \cite{DFH}, Dinh, Hajaiej, and the second author proved the following.  
	\begin{theorem} \label{theo-blow-crite}
		Let $\lambda_1$ and $\lambda_2$ satisfy \eqref{UR}. Let $u(t)$ be a solution to \eqref{GP}, defined on the maximal forward time interval $[0,T_{max})$. Assume that there exists a positive constant  $\delta>0$ such that 
\begin{equation} \label{blow-crite}
\sup_{t\in [0,T_{max})} G(u(t)) \leq -\delta.
\end{equation}
Then either the maximal forward time $T_{max}<\infty$, or $T_{max}=\infty$ and there exists a diverging sequence of times, say $t_n \rightarrow \infty$ as $n\to\infty$, such that $ \lim_{n\to\infty}\| u(t_n)\|_{\dot H^1(\R^3)}= \infty $.  In the latter case we say that the solution grows-up.
\end{theorem}

The next corollary actually shows that the condition given in \autoref{theo-blow-crite} is non-empty, as an initial datum belonging to the region $(BC)$, see \eqref{blow-reg}, leads to  solution satisfying \eqref{blow-crite} (see our paper \cite[Section 3]{BF21}).
\begin{corollary} 
Let $\lambda_1$ and $\lambda_2$ satisfy \eqref{UR}, and $Q$ be a ground state related to \eqref{ell-equ}. Assume that $u_0 \in H^1(\R^3)$ satisfies \eqref{blow-reg}
\noindent and let $u(t)$ the corresponding solution to \eqref{GP}. Then either $T_{max}<\infty$, or $T_{max}=\infty$ and $u(t)$ grows-up.
\end{corollary}

\subsubsection{Dynamics above the threshold} For the dynamical properties of solutions to \eqref{GP} above the threshold, we need to further restrict the Unstable Regime, and we introduce the Restricted Unstable Regime as follows:
\begin{equation}\label{cond-GW}
RUR:=\left\{ \begin{aligned}
\lambda_1 +\frac{8\pi}{3} \lambda_2 <0&\quad \hbox{ if } \quad \lambda_2>0\\
\lambda_1-\frac{4\pi}{3} \lambda_2 <0  & \quad\hbox{ if } \quad\lambda_2<0
\end{aligned}\right..
\end{equation}
For a ground state  $Q$ related to \eqref{ell-equ}, we also give the scattering or blow-up conditions above the threshold:
\begin{equation}\label{above:sca}
(SC^\prime):=\left\{ 
\begin{aligned}
E(u_0)M(u_0) &\geq  E(Q) M(Q) \\
 \frac{E(u_0)M(u_0)}{E(Q)M(Q)} &\left(1-\frac{(V'(0))^2}{8E(u_0)V(0)}\right) \leq 1\\
		-P(u_0)M(u_0) &< -P(Q)M(Q) \\
		V'(0) &\geq 0 
\end{aligned}\right.,
\end{equation}
and
\begin{equation}\label{above:b-up}
(BC^\prime):=\left\{ 
\begin{aligned}
E(u_0)M(u_0) &\geq  E(Q) M(Q) \\
 \frac{E(u_0)M(u_0)}{E(Q)M(Q)} &\left(1-\frac{(V'(0))^2}{8E(u_0)V(0)}\right) \leq 1 \\
-P(u_0)M(u_0) &> -P(Q)M(Q)\\
V'(0) &\leq 0 
\end{aligned}\right.,
\end{equation}
respectively. Initial data satisfying \eqref{above:sca} or \eqref{above:b-up} can be constructed by a simple scaling argument, see \cite{DFH} and \cite{GW}.\\

The following blow-up result above the threshold has been given by Gao and Wang in \cite{GW}.	
	\begin{theorem} \label{theo-blow-GW}
		Let $\lambda_1$ and $\lambda_2$ satisfy \eqref{cond-GW}. 
				Let $Q$ be a ground state related to \eqref{ell-equ}, and $u_0 \in \Sigma$ satisfy \eqref{above:b-up}. Then the corresponding solution $u(t)$ to \eqref{GP} blows-up forward in finite time.
	\end{theorem}
The counterpart of \autoref{theo-blow-GW} is the following scattering result, given for initial data satisfying \eqref{above:sca}. It is one of the main theorems contained in the paper by Dinh, Hajaiej, and the second author \cite{DFH}.

		\begin{theorem} \label{theo-scat-above}
		Let $\lambda_1$ and $\lambda_2$ satisfy \eqref{cond-GW}. Let $Q$ be a ground state related to \eqref{ell-equ}, and $u_0 \in \Sigma$ be such that \eqref{above:sca} holds true.	 Then the corresponding solution $u(t)$ to \eqref{GP} exists globally and scatters in $H^1(\R^3)$ forward in time.
	\end{theorem}

Concerning the Theorems above, it is worth mentioning the reason why we have to consider the subset $RUR$ of the Unstable Regime, see \eqref{cond-GW}, instead of the whole configurations of the parameters $\lambda_1$ and $\lambda_2$ as in \eqref{UR}.	Conditions \eqref{cond-GW} imply a control on the potential energy sign, specifically it is negative for any time along the evolution of the solution. This will play a crucial role in the proof of the scattering criterion \autoref{theo-scat-crite} below.

\subsubsection{Dynamics at the threshold}
The next Theorem deals with  the long time dynamics for solutions  to \eqref{GP} at the Mass-Energy threshold, i.e. when the initial datum satisfies
\begin{equation} 
		E(u_0)M(u_0) = E(Q) M(Q). \label{cond-ener-at}
		\end{equation}
In \cite{DFH}, Dinh, Hajaiej, and the second author, gave a complete picture of the dynamics under the hypothesis \eqref{cond-ener-at}, by analysing different scenario described in terms of the quantity $H(u_0) M(u_0)$. To the best of our knowledge, early results on for the focusing cubic NLS at the threshold are given in the work of Duyckaerts and Roudenko \cite{DR1}. The Theorem is as follows.
	\begin{theorem}\label{theo-dyna-at}
		Let $\lambda_1$ and $\lambda_2$ satisfy \eqref{UR}. Let $Q$ be a ground state related to \eqref{ell-equ}. Suppose that  $u_0 \in H^1(\R^3)$ satisfies the Mass-Energy threshold condition \eqref{cond-ener-at}.	 We have the following three scenarios. \\\medskip
			\noindent \textup{(i)}	In addition to \eqref{cond-ener-at}, suppose that  
			\begin{align} \label{cond-scat-at}
			H(u_0) M(u_0) < H(Q) M(Q),
			\end{align}
			and that the  corresponding solution $u(t)$ to \eqref{GP} is defined on the maximal interval of existence $(-T_{min}, T_{max})$. Then for every $t\in(-T_{min}, T_{max})$
			\begin{equation*} \label{est-solu-at-1}
			H(u(t)) M(u(t)) < H(Q) M(Q)
			\end{equation*}
			and in particular $T_{min}=T_{max}=\infty$. Moreover, provided $\lambda_1$ and $\lambda_2$ satisfy \eqref{cond-GW}, then the solution
			\begin{itemize}
			\item either scatters in $H^1(\R^3)$ forward in time, 
			\item or there exist a diverging sequence of times $t_n\rightarrow \infty$ as $n\to\infty$, a ground state $\tilde{Q}$ related to \eqref{ell-equ}, and a  sequence $\{y_n\}_{n\geq1}\subset\R^3$ such that for some $\theta \in \R$ and $\mu>0$
			\begin{equation} \label{conver-tn}
			u(t_n, \cdot-y_n) \rightarrow e^{i\theta} \mu \tilde{Q}(\mu\cdot) 			\end{equation}
			strongly in $H^1(\R^3)$ as $n\rightarrow \infty$.
			\end{itemize}
			\medskip
			\noindent \textup{(ii)} In addition to \eqref{cond-ener-at}, suppose that
			\begin{equation} \label{cond-at}
			H(u_0) M(u_0)= H(Q) M(Q),
			\end{equation}
			then there exists a ground state $\tilde{Q}$ related to \eqref{ell-equ} such that the solution $u(t)$ to \eqref{GP} satisfies $u(t,x) = e^{i\mu^2t} e^{i\theta} \mu \tilde{Q}(\mu x)$ for some $\theta\in \R$ and $\mu>0$, and hence the solution is global.\medskip\\ 
			\noindent  \textup{(iii)} In addition to \eqref{cond-ener-at}, suppose that
			\begin{align} \label{cond-blow-at}
			H(u_0) M(u_0)> H(Q) M(Q),
			\end{align}
			and that the  corresponding solution $u(t)$ to \eqref{GP} is defined on the maximal interval of existence $(-T_{min}, T_{max})$.	Then for every $t\in(-T_{min}, T_{max})$
		\begin{equation*} \label{est-solu-at-3}
			H(u(t)) M(u(t)) > H(Q) M(Q).
			\end{equation*}
 Furthermore, the solution 
\begin{itemize}
\item either blows-up forward in finite time,
\item or it grows-up along some diverging sequence of times $t_n\rightarrow \infty$ as $n\to \infty$, 
\item or there exists a diverging sequence of times $t_n \rightarrow \infty$ as $n\to \infty$ such that \eqref{conver-tn} holds for some sequence $\{y_n\}_{n\geq1}\subset\R^3,$ and some parameters  $\theta\in \R,$ and  $\mu>0.$
\end{itemize}
\noindent Provided $u_0\in \Sigma$, then the grow-up scenario as in the second point  is ruled out.
\end{theorem}

\section{Decay for powers of  Riesz transforms and virial arguments}\label{sec-riesz}

This section provides the first technical tools we need in order to prove our main results. Moreover, we present the strategy  we adopt to prove the main theorems, which strongly rely on virial arguments based on the decay for powers of Riesz transforms that we are going to prove. \\

First of all, we recall the fact that the dipolar kernel defines a Calder\'on-Zygmund operator, hence it is a well-kwon fact that it yields to a map continuous from $L^p$ into itself, for non end-point Lebesgue exponents, namely for $p\neq1$ and $p\neq \infty$. For a proof  see \cite[Lemma 2.1]{CMS}.
\begin{prop}\label{prop:cont-K} The convolution operator $f\mapsto K\ast f$ can be extended as a continuous operator from $L^p$ into itself, for any $p\in(1,\infty)$.
\end{prop}

Moreover, in \cite{CMS}, an explicit computation of the Fourier transform of the dipolar kernel $K$ defined in \eqref{eq:ker} is given. Precisely, we have the following.
\begin{lemma}\label{lemma:fou-k} The Fourier transform of the dipolar kernel $K$ is given by:
\begin{equation}\label{kernel:fou}
\hat K(\xi)=\frac{4\pi}{3}\frac{2\xi_3^2-\xi_2^2-\xi_1^2}{|\xi|^2}, \qquad \xi\in\R^3.
\end{equation}
Straightforwardly, it follows that $\hat K\in\left[-\frac43\pi,\frac83\pi\right]$.
\end{lemma}
\noindent For a proof of \eqref{kernel:fou}, we refer  to  \cite[Lemma 2.3]{CMS}. The explicit calculation of $\hat K,$ is done by means of the decomposition in spherical harmonics of the Fourier character $e^{-ix\cdot\xi}.$ \\

\subsection{Integral estimates for $\mathcal R^4_j$} 
In the next Propositions, we prove some decay estimates -- point-wise and integral ones -- regarding the square and the fourth power of the Riesz transforms when acting on suitably localized functions. Firstly, we disclose a link between the fourth power of the Riesz transform $\mathcal R_j^4$ and  the linear propagator associated to the parabolic biharmonic equation, defined in terms of the Bessel functions. With this correspondence and some decay estimates for the parabolic biharmonic heat kernel we are able to show the decay estimate for $\langle \mathcal R^4_jf, g\rangle.$ Here $\langle \cdot, \cdot \rangle$ stands for the usual $L^2(\R^3)$ inner product.
We start with the integral estimates for the fourth power of the Riesz transform, and, as anticipated above,  we do it by means of some decay properties of the kernel associated to the parabolic biharmonic equation
\begin{equation}\label{eq:biha}
\partial_t w+\Delta^2w=0, \qquad (t,x)\in\R\times\R^3.
\end{equation}
We denote by $P_t$ the linear propagator associated to \eqref{eq:biha}, namely $w(t,x):=P_tw_0(x)$ denotes the solution to the equation \eqref{eq:biha} with initial datum $w_0.$ We begin with the following proposition which provides a representation of $\mathcal R_j^4$ by using the functional calculus. Since now on, we will omit -- unless necessary -- the notation $\R^3$, as we are concerned only with the three-dimensional model.  
\begin{lemma}\label{biharm}
For any two functions in $L^2$ we have the following identity:
\begin{equation}\label{eq:r4}
\langle \mathcal R_i^4f,g\rangle=-\int_0^\infty\langle \partial_{x_i}^4\frac{d}{dt}P_tf,g\rangle t\,dt.
\end{equation}
\end{lemma}
\begin{proof}
By passing in the frequencies space, it is easy to see that $\widehat {P_tf}(\xi):=e^{-t|\xi|^4}\hat f(\xi)$ and we observe, by integration by parts, that 
\begin{equation}\label{eq:id4}
\xi^4_i|\xi|^4\int_0^\infty e^{-t|\xi|^4}t\,dt=\frac{\xi^4_i}{|\xi|^4};
\end{equation}
hence
\[
\begin{aligned}
\int_0^\infty\langle \partial_{x_i}^4\frac{d}{dt}P_tf,g\rangle t\,dt&=\langle\int_0^\infty \partial_{x_i}^4\frac{d}{dt}(P_tf)t\,dt, g \rangle=\langle\int_0^\infty \xi_i^4\frac{d}{dt}(e^{-t|\xi|^4}\hat f)t\,dt, \hat g\rangle\\
&=- \langle \xi_i^4|\xi|^4\hat f\int_0^\infty e^{-t|\xi|^4} t\,dt,\hat g\rangle= -\langle \frac{\xi_i^4}{|\xi|^4}\hat f,\hat g\rangle=-\langle \mathcal R_i^4f,g\rangle,
\end{aligned}
\]
where the change of order of integration (in time and in space) is justified by means of the Fubini-Tonelli's theorem, and we used the Plancherel identity when passing from the frequencies space to the physical space, and vice versa . 

\end{proof}
We are now in position to prove a decay estimate for functions supported outside a cylinder of radius $\gtrsim R.$ In order to do that, we explicitly write the heat kernel of $P_t.$ We introduce, for $t>0$ and $x\in\R^3$ 
\[
p_t(x)=\alpha\frac{k(\mu)}{t^{3/4}}, \quad \mu=\frac{|x|}{t^{1/4}},
\]
and
\[
k(\mu)=\mu^{-2}\int_0^\infty e^{-s^4}(\mu s)^{3/2}J_{1/2}(\mu s)\,ds,
\]
where $J_{1/2}$ is the $\frac12$-th Bessel function, and $\alpha^{-1}:=\frac{4\pi}{3}\int_0^\infty s^2k(s)\,ds$ is a positive normalization constant. We refer to  \cite{FGG} for these definitions and further discussions about the heat kernel of  the parabolic biharmonic equation. 
We recall that the $\frac12$-th Bessel function is given by 
\[J_{1/2}(s)=(\pi/2)^{-1/2} s^{-1/2}\sin(s),
\] 
then  
\[
P_t f(x)=(p_t\ast f)(x)=c\int f(x-y)\int_0^\infty\frac{1}{|y|^3}e^{-ts^4/|y|^4} s\sin{(s)}\,ds\,dy,
\]
and therefore 
\[
\frac{d}{dt}P_t f(x)=-c\int f(x-y)\int_0^\infty\frac{1}{|y|^3}e^{-ts^4/|y|^4} \frac{s^5}{|y|^4}\sin{(s)} \,ds\,dy.
\]
We are ready to prove the following result.
\begin{prop}\label{lemma:decay-r4}
Assume that $f,g\in L^1\cap L^2,$ and that $f$ is supported in $\{|\bar x|\geq \gamma_2R\}$ while $g$ is supported in $\{|\bar x|\leq \gamma_1R\},$ for some positive parameters $\gamma_{1,2}$ satisfying $d:=\gamma_2-\gamma_1>0.$ Then 
\begin{equation}\label{eq:est-r4}
|\langle \mathcal R_i^4f,g\rangle|\lesssim R^{-1}\|g\|_{L^1}\|f\|_{L^1}.
\end{equation}
\end{prop}
\begin{proof}
With the change of variable $s^4|y|^{-4}=\tau$ we get 
\[
\frac{d}{dt}P_t f=-\frac c4\int\int_0^\infty\frac{1}{|y|}e^{-t\tau} \tau^{1/2}\sin(\tau^{1/4}|y|)f(x-y)\,d\tau\,dy
\]
and hence, by a change of variable in space, 
\[
\frac{d}{dt}P_t f=-\frac c4\int\int_0^\infty\frac{1}{|x-y|}e^{-t\tau} \tau^{1/2}\sin(\tau^{1/4}|x-y|)f(y)\,d\tau\,dy.
\]
We will use the following, by adopting the notation  $deg$  for the degree of a polynomial.
\begin{claim}
There exist $M\geq 1$ and $M$ pairs of polynomials $(\tilde q_k,q_k)_{k\in\{1,\dots, M\}}$ with nonnegative coefficients, such that 
\[
\min_{k\in\{1,\dots, M\}}\{deg (q_k)\}\geq 1,
\] 
and satisfying 
\begin{equation*}\label{eq:sin}
\left|\partial_{x_i}^4\left(\frac{1}{|x-y|} \sin(\tau^{1/4}|x-y|)\right)\right|\lesssim \sum_{k=1}^M\frac{\tilde q_k(\tau^{1/4})}{q_k(|x-y|)}.
\end{equation*}
\end{claim}
At this point, by using the identity \eqref{eq:r4} we infer the following:
\[
\begin{aligned}
|\langle \mathcal R_i^4f,g\rangle|&=\left|\int_0^\infty\langle \partial_{x_i}^4\frac{d}{dt}P_tf,g\rangle t\,dt\right|\\
&=c\left| \int_0^\infty t \int g(x)\left( \int\int_0^\infty \partial_{x_i}^4\left(\frac{1}{|x-y|} \sin(\tau^{1/4}|x-y|)\right)e^{-t\tau}\tau^{1/2}f(y)\,d\tau\,dy\right) \,dx\,dt\right|\\
&\lesssim \int_0^\infty t \int |g(x)|\left( \int \int_0^\infty \sum_{k=1}^M\frac{\tilde q_k(\tau^{1/4})}{q_k(|x-y|)}e^{-t\tau}\tau^{1/2}|f(y)|\,d\tau\,dy\right) \,dx\,dt\\
&=\int_0^\infty t \int |g(x)|\left( \int \int_0^\infty \sum_{k=1}^M\frac{\tilde q_k(\tau^{1/4})}{q_k(|y|)}e^{-t\tau}\tau^{1/2}|f(x-y)|\,d\tau\,dy\right) \,dx\,dt\\
&\leq\int_0^\infty t \int_{\{|\bar x|\leq\gamma_1R\}}|g(x)|\left( \int_{\{|\bar x-\bar y|\geq \gamma_2R\}}\int_0^\infty \sum_{k=1}^M\frac{\tilde q_k(\tau^{1/4})}{q_k(|\bar y|)}e^{-t\tau}\tau^{1/2}|f(x-y)|\,d\tau\,dy\right) \,dx\,dt.
\end{aligned}
\]
Therefore, as the support of $f(x-y)$ is contained in $|\bar x-\bar y|\geq \gamma_2R$ and the one of $g$ is contained in $|\bar x|\leq \gamma_1R,$ we get that $|\bar y|\geq dR.$ Hence, by defining $\beta=\frac14\max_{k\in\{1,\dots,M\}}\{deg(\tilde q_k))\},$ we can bound 
\[
\begin{aligned}
|\langle \mathcal R_i^4f,g\rangle|&\lesssim R^{-1}\|f\|_{L^1}\|g\|_{L^1} \int_0^\infty \int_0^\infty e^{-t\tau}\tau^{1/2}\tilde q_k(\tau^{1/4})\, t\,d\tau\,dt\\
&\sim R^{-1}\|f\|_{L^1}\|g\|_{L^1}\int_1^\infty\int_1^\infty e^{-t\tau}\tau^{1/2}\tau^{(\max_{k\in\{1,\dots,M\}}deg(\tilde q_k))/4}\, t\,d\tau \, dt\\
&= R^{-1}\|f\|_{L^1}\|g\|_{L^1}\int_1^\infty\int_1^\infty e^{-t\tau}\tau^{\beta+1/2}\, t\,d\tau \, dt\\
&\lesssim R^{-1}\|f\|_{L^1}\|g\|_{L^1}
\end{aligned}
\]
where we used the Fubini-Tonelli's theorem. The proof of \eqref{eq:est-r4} is concluded. 
\end{proof}
We  now give the proof of the Claim above. 
\begin{proof}[Proof of Claim 2.3]
As the derivative is invariant under translations and by defining $c=\tau^{1/4},$  we can reduce everything to the estimate of $\partial_{x_i}^4\left(|x|^{-1}\sin(c|x|)\right)$. By setting $f(r)=r^{-1}\sin(cr)$ and $g(x)=|x|$ we can see 
\[
|x|^{-1}\sin(c|x|)=(f\circ g)(x),
\]
and without loss of generality we assume $i=3.$ Then we see $g$ as a function of $x_3$ alone,  i.e. $g(x_3)=\left(x^2_1+x_2^2+x_3^2\right)^{1/2}$.
We first collect some identities. 
\begin{equation*}
\begin{aligned}
f^\prime(r)&=cr^{-1}\cos(cr)-r^{-2}\sin(cr)\\
f^{\prime\prime}(r)&=-c^2r^{-1}\sin(cr)-2cr^{-2}\cos(cr)+2r^{-3}\sin(cr)\\
f^{\prime\prime\prime}(r)&=-c^3r^{-1}\cos(cr)+c^2r^{-2}\sin(cr)+8cr^{-3}\cos(cr)-6r^{-4}\sin(cr)\\
f^{\prime\prime\prime\prime}(r)&=c^4r^{-1}\sin(cr)+2c^3r^{-2}\cos(cr)-10c^2r^{-3}\sin(cr)-30cs^{-4}\cos(cr)+24r^{-5}\sin(cr)
\end{aligned}
\end{equation*}
and 
\begin{equation*}
\begin{aligned}
g^\prime=\partial_{x_3}g(x_3)&=\frac{x_3}{(x_1^2+x_2^2+x_3^2)^{1/2}},\\
g^{\prime\prime}=\partial_{x_3}^2g(x_3)&=\frac{1}{(x_1^2+x_2^2+x_3^2)^{1/2}}-\frac{x_3^2}{(x_1^2+x_2^2+x_3^2)^{3/2}},\\
g^{\prime\prime\prime}=\partial_{x_3}^3g(x_3)&=-\frac{3}{(x_1^2+x_2^2+x_3^2)^{3/2}}+\frac{3x_3^3}{(x_1^2+x_2^2+x_3^2)^{5/2}},\\
g^{\prime\prime\prime\prime}=\partial_{x_3}^4g(x_3)&=-\frac{3}{(x_1^2+x_2^2+x_3^2)^{3/2}}+\frac{18x_3^2}{(x_1^2+x_2^2+x_3^2)^{5/2}}-\frac{15x_3^4}{(x_1^2+x_2^2+x_3^2)^{7/2}}.
\end{aligned}
\end{equation*}
At this point we recall that by the Fa\`a di Bruno's formula 
\[
\begin{aligned}
\partial_{x_3}^4(f\circ g)(x)&=f^{\prime\prime\prime\prime}(|x|)[g^\prime(x)]^4+6f^{\prime\prime\prime}(|x|) g^{\prime\prime}(x)[g^{\prime}(x)]^2+3f^{\prime\prime}(|x|)[g^{\prime\prime}(x)]^2\\
&+4f^{\prime\prime}(|x|) g^{\prime\prime\prime}(x)g^{\prime}(x)+f^\prime(|x|)g^{\prime\prime\prime\prime}(x)
\end{aligned}
\]
and the Claim  easily follows by replacing  $c=\tau^{1/4}$ and translating back $x\mapsto x-y.$ 
\end{proof}

\begin{remark}\label{rmk:2-2}
It is straightforward to observe that in \eqref{eq:id4} we can replace the symbol $\xi_j^4$ with $\xi_k^2\xi_h^2,$ for $k\neq h,$ to get 
\begin{equation}\label{id:r2-2}
\frac{\xi_k^2\xi_h^2}{|\xi|^4}=\xi_k^2\xi_h^2|\xi|^4\int_0^\infty e^{-t|\xi|^4}t\,dt,
\end{equation}
and consequently 
\begin{equation}\label{eq:r2-2}
\langle \mathcal R^2_k\mathcal R_h^2 f,g\rangle=-\int_0^\infty\langle \partial_{x_k}^2\partial_{x_h}^2\frac{d}{dt}P_tf,g\rangle t\,dt
\end{equation}
\end{remark}

The identities \eqref{id:r2-2} and \eqref{eq:r2-2} of the remark above easily imply an analogous of  \autoref{lemma:decay-r4} (by repeating its proof with the obvious modifications) for the operator $\mathcal R^2_k\mathcal R_h^2$ replacing $\mathcal R^4_j.$ More precisely:
\begin{prop}\label{lemma:decay-r2-2}
Assume that $f,g\in L^1\cap L^2,$ and that $f$ is supported in $\{|\bar x|\geq \gamma_2R\}$ while $g$ is supported in $\{|\bar x|\leq \gamma_1R\},$ for some $\gamma_{1,2}>0$ satisfying $d:=\gamma_2-\gamma_1>0.$ Then 
\begin{equation*}
|\langle \mathcal R^2_k\mathcal R_h^2 f,g\rangle|\lesssim R^{-1}\|g\|_{L^1}\|f\|_{L^1}.
\end{equation*}
\end{prop}

\subsection{Pointwise estimate for $\mathcal R^2_j$.}
We turn now the attention to the square of the Riesz transform. In the subsequent results, we will use a cut-off function $\chi$ satisfying the following: $\chi(x)$ is a localization function supported in the cylinder $\{|\bar x|\leq 1\}$  which is nonnegative and bounded, with $\|\chi\|_{L^\infty}\leq 1.$ For a positive parameter $\gamma,$ we define by $\chi_{\{|\bar x|\leq \gamma R\}}$ the rescaled function $\chi(x/\gamma R)$ (hence $\chi_{\{|\bar x|\leq \gamma R\}}$ is bounded, positive and supported in the cylinder of radius $\gamma R$). The proof of the next lemmas is inspired by \cite{LW}.

\begin{prop}\label{lemma:in-out} For any (regular) function $f$ the following point-wise estimate is satisfied: provided $d:=\gamma_2-\gamma_1>0,$ there exists an universal constant $C=C(d)>0$ such that 
\begin{equation}\label{eq:in-out}
|\chi_{\{|\bar x|\leq \gamma_1R\}}(x)\mathcal R_j^2[(1-\chi_{\{|\bar x|\leq \gamma_2R\}})f](x)|\leq CR^{-3}\chi_{\{|\bar x|\leq \gamma_1R\}}(x)\|f\|_{L^1(|\bar x|\geq\gamma_2R)}.
\end{equation}
\end{prop}
We have an estimate similar to \eqref{eq:in-out} if we localize inside a cylinder the function on which $R^2_j$ acts, and we then truncate everything with a function supported in the exterior of  another cylinder. 
\begin{prop}\label{lemma:out-in} For any (regular) function $f$ the following point-wise estimate is satisfied: provided $d:=\gamma_1-\gamma_2>0,$ there exists an universal constant $C=C(d)>0$ such that 
\begin{equation*}\label{eq:out-in}
|(1-\chi_{|\bar x|\leq \gamma_1R})(x)\mathcal R_j^2[(\chi_{\{|\bar x|\leq \gamma_2R\}})f](x)|\leq CR^{-3}|(1-\chi_{\{|\bar x|\leq \gamma_1R\}})(x)|\|f\|_{L^1(|\bar x|\leq \gamma_2R)}.
\end{equation*}
\end{prop}

\begin{proof} The proofs of the Lemmas above are analogous, and they can be given by observing that in the principal value sense, the square of the Riesz transform acts on a function $g$ as 
\[
\mathcal R_j^2g(x)=\iint \frac{x_j-y_j}{|x-y|^{3+1}}\frac{y_j-z_j}{|y-z|^{3+1}}g(z)\,dz\,dy.
\]
Without loss of generality, we consider the case depicted in  \autoref{lemma:in-out}. Let $g(x)=\chi_{\{|\bar x|\geq\gamma_2R\}}(x)f(x).$ Then
\[
\chi_{\{|\bar x|\leq \gamma_1R\}}(x)\mathcal R_j^2g(x)=\iint\left(\frac{y_j}{|y|^{4}}\frac{z_j-y_j}{|z-y|^{4}}\,dy\right)g(x-z)\,dz.
\] 
Since $g$ is supported in the exterior of a cylinder of radius $\gamma_2R,$ we can assume  $|\bar x-\bar z|\geq \gamma_2R,$ and for the function $\chi_{\{|\bar x|\leq \gamma_1R\}}$ is supported by definition in the cylinder of radius $\gamma_1R,$ we can assume  $|\bar x|\leq \gamma_1R:$ therefore we have that $|\bar z|\geq dR.$ This implies that $\{|\bar y|\leq \frac{d}{4}R\}\cap\{|\bar z-\bar y|\leq\frac12|\bar z|\}=\emptyset.$ Indeed, 
\begin{equation*}\label{eq:y}
\frac12|\bar z|\geq|\bar z-\bar y|\geq |\bar z|-|\bar y| \implies |\bar y|\geq\frac12|\bar z|\geq \frac d2R,
\end{equation*}
hence we have the following splitting for the inner integral:
\begin{equation*}\label{eq:int:I}
\begin{aligned}
\int\frac{y_j}{|y|^{4}}\frac{z_1-y_1}{|z-y|^{4}}\,dy&=\int_{|\bar y|\leq\frac d4 R}\frac{y_j}{|y|^{4}}\frac{z_j-y_j}{|z-y|^{4}}\,dy
+\int_{|\bar z-\bar y|\leq\frac12|\bar z|}\frac{y_j}{|y|^{4}}\frac{z_j-y_j}{|z-y|^{4}}\,dy\\
&+\int_{\{|\bar y|\geq \frac d4 R\}\cap\{|\bar z-\bar y|\geq \frac12|\bar z|\}}\frac{y_j}{|y|^{4}}\frac{z_j-y_j}{|z-y|^{4}}\,dy\\
&=\mathcal A+\mathcal B+\mathcal C.
\end{aligned}
\end{equation*}
By using the properties of the domains in this splitting, the proof of the Lemma can be done by straightforward computations, ending up with
\[
\int\frac{y_j}{|y|^{4}}\frac{z_1-y_1}{|z-y|^{4}}\,dy=\mathcal A+\mathcal B+\mathcal C\lesssim R^{-3},
\]
and hence
\[
\begin{aligned}
|\chi_{\{|\bar x|\leq \gamma_1R\}}(x)\mathcal R_j^2g(x)|&=\chi_{\{|\bar x|\leq \gamma_1R\}}(x)\left|\iint\left(\frac{y_j}{|y|^{4}}\frac{z_j-y_j}{|z-y|^{4}}\,dy\right)g(x-z)\,dz\right|\\
&\lesssim R^{-3}\chi_{\{|\bar x|\leq \gamma_1R\}}(x)\int|g(x-z)|\,dz\\
&\lesssim R^{-3}\chi_{\{|\bar x|\leq \gamma_1R\}}(x)\|f\|_{L^1(|\bar x|\geq\gamma_2R)}
\end{aligned}
\] 
which is the estimate stated in \eqref{eq:in-out}. See \cite{BF21} for the details.
\end{proof}

The proofs of \autoref{lemma:decay-r4}, \autoref{lemma:decay-r2-2}, \autoref{lemma:in-out}, and \autoref{lemma:out-in} can be done by using an alternative approach, by means of a general characterization of homogeneous distribution on $\R^n$ of degree $-n,$ coinciding with a regular function in $\R^n\setminus\{0\}.$ Indeed, we have the following (we specialise to the three-dimensional case). For a proof, see \cite{BF21}. In what follows,  `$\hbox{dist}$' denotes the distance function.
\begin{prop}\label{thm:zero-degree}
Let $T$ an operator defined by means of a Fourier symbol $m(\xi),$ which is smooth in $\R^3\setminus\{0\}$ and is a homogenous function of degree zero, i.e. $m(\lambda \xi)=m(\xi)$ for any $\lambda>0.$ For any couple of functions $f,g\in L^1$ having disjoint supports, we have the following estimate:
\begin{equation*}\label{eq:int-decay-general}
|\langle Tf,g\rangle|\lesssim \left(\mbox{dist}(\mbox{supp}(f), \mbox{supp}(g))\right)^{-3}\|g\|_{L^1}\|f\|_{L^1}.
\end{equation*}
\end{prop}
\begin{remark} Keeping in mind the general statement of \autoref{thm:zero-degree}, it is easy for the reader to see  that similar results as in \autoref{lemma:decay-r4}, \autoref{lemma:decay-r2-2}, \autoref{lemma:in-out}, and \autoref{lemma:out-in} can be stated for functions localized outside and inside disjoint balls, instead of disjoint cylinders. Such a localizations for functions supported outside and inside balls will be used for the scattering results using a concentration/compactness and rigidity scheme.
\end{remark}

\subsection{Virial identities}\label{subsec:vir}
The main difference between the Gross-Pitaevskii equation \eqref{GP} and the classical cubic NLS equation, is the non-local character of the nonlinearity, in conjunction  with the fact that the kernel $K$ requires  a more careful treatment with respect to the usual Coulomb or Hartree type kernels.  Hence we spend a few words here, to give an overview on how the  results concernig the decay of powers of the Riesz transforms as in the previous subsections, will play a central role in the  proofs of the  main theorems. We show below how the tools above will be used for both the scattering and blow-up/grow-up results.  \medskip\\
\noindent \textup{(i)} Standard arguments show that, provided \eqref{sc-reg} is satisfied, then the Pohozaev functional $G$ is bounded from below uniformly in time, in particular  there exists a positive $\alpha$ such that $G(u(t))\geq\alpha>0$ for all times in the maximal interval of existence of the solution. Similarly, provided \eqref{blow-reg} holds true, then  $G(u(t))\leq-\delta<0$ for all times in the maximal interval of existence of the solution, for some positive $\delta$. As a byproduct,  $G(u(t))\lesssim -\delta \|u(t)\|_{\dot H^1}^2$.  \medskip\\
\noindent \textup{(ii)} Let $\chi$ a (regular) nonnegative function, which will be well-chosen below. Let us denote by $\chi_R$ the rescaled version of $\chi$, defined by $\chi_R=R^2\chi(x/R)$, and let us introduce the quantity
\begin{equation}\label{vir-1}
V_{\chi_R}(t):=V_{\chi_R}(u(t))=2\int \chi_R(x)|u(t,x)|^2\,dx.
\end{equation}
By formal computations, which can be justified by a classical regularization argument, it is easy to show that 
\begin{equation}\label{eq:intro-vir}
\frac{d^2}{dt^2}V_{\chi_R}(t)=4\int|\nabla u(t)|^2 dx +6 \lambda_1\int|u(t)|^4 dx+H_R(u(t)),
\end{equation}
where the error  $H_R$ is an error term which must be controlled. Our aim is to show that 
\begin{equation}\label{remainder-1}
H_R(u(t))=6\lambda_2\int(K\ast|u(t)|^2)|u(t)|^2\,dx + \epsilon_R
\end{equation}
where $\epsilon_R=o_R(1)$ as $R\to\infty$, uniformly in time in the lifespan of the solution. Let us observe that by glueing together \eqref{eq:intro-vir} and \eqref{remainder-1} we get, by recalling the definition of $G$, see \eqref{def:G},
\[
\frac{d^2}{dt^2}V_{\chi_R}(t)=4G(u(t))+\epsilon_R.
\]
By using the controls on the Pohozaev functional as described in the first  point, and provided that $\epsilon_R$ is made sufficiently small for some $R$ large enough, then we are able to conclude the concentration/compactness and rigidity scheme for the scattering results, or with can close estimates for a convexity argument for the blow-up results. See the next points and the discussions in the next sections.\medskip\\ 
\noindent \textup{(iii-a)} As for the scattering part, let us mention for sake of clarity that the aim  of the concentration/compactness and rigidity scheme is to prove that all solutions arising from initial data satisfying \eqref{sc-reg} are global and scatter. Let us recall that small initial data lead to global and scattering solutions, by a standard perturbative argument. The Kenig and Merle's road map (see \cite{KM1, KM2}) then proceeds as follows: suppose that the threshold for global and scattering solutions is strictly smaller than the claimed one (i.e. $E(Q)M(Q)$); then, by means of a profile decomposition Theorem,  it is possible to construct a minimal global non-scattering solution at  the threshold energy. Moreover such a solution, called soliton-like solution and denoted by $u_{crit}$, is precompact in the energy space  up to a continuous-in-time translation path $x(t)$, i.e. $\{u(t,x+x(t))\}_{t\in\R^+}$ is precompact in $H^1$. The crucial fact is that such a path $x(t)$ grows sub-linearly at infinity, and this will rule out the existence of such a soliton-like solution. This latter fact is proved by using the precompactness of the (translated) flow in conjunction with a virial argument, along with the already mentioned growth property of $x(t)$. 

For the virial argument in this context, we choice $\chi$ to be a cut-off function such that $\chi(x)=|x|^2$ on $|x|\leq1$ and $supp\,(\chi)\subset B(0,2)$ (namely, we consider a localized version  of \eqref{defi-V}). We get 
\begin{equation}\label{vir-3}
\begin{aligned}
\frac{d^2}{dt^2}V_{\chi_R}(t)&=4\int |\nabla u(t)|^2\,dx+6\lambda_1\int |u(t)|^4\,dx +\varepsilon_{1,R}\\
&-2\lambda_2R\int\nabla\chi\left(\frac xR\right)\cdot\nabla\left(K\ast|u(t)|^2\right)|u(t)|^2\,dx,\
\end{aligned}
\end{equation}
where 
\begin{equation}\label{rem-eps1}
\varepsilon_{1,R} = C\left(\int_{|x|\geq R}|\nabla u(t)|^2+R^{-2}|u(t)|^2+|u(t)|^4\,dx\right).
\end{equation}
The quantity $\varepsilon_1(R)$ can be made small, uniformly in time, for $R$ sufficiently large, by using the precompactness of the soliton-like solution constructed with the concentration/compactness scheme. To handle $\Lambda:=-2\lambda_2R\int\nabla\chi\left(\frac xR\right)\cdot\nabla\left(K\ast|u(t)|^2\right)|u(t)|^2\,dx$  in \eqref{vir-3}, we perform a splitting in space of the solution $u(t,x)$  by considering its cut-off inside and outside a ball of radius $\sim R$, eventually obtaining the identity   $\Lambda=6\lambda_2\int(K\ast|u(t)|^2)|u(t)|^2\,dx + \varepsilon_{2,R}$.
It is after the splitting above that we can reduce to a term $\varepsilon_{2,R}$  which fulfils the hypothesis of the point-wise decay of the Riesz transforms as in the previous Section; indeed, with such a localized functions, we can lead back our term $\varepsilon_{2,R}$ in the framework of \autoref{lemma:in-out} and \autoref{lemma:out-in}.   By letting $\epsilon_R:=\varepsilon_{1,R}+\varepsilon_{2,R}$, for $R$ sufficiently large, we get 
\[
\frac{d^2}{dt^2}V_{\chi_R}(t)=4G(u(t))+\epsilon_R\geq 2\alpha,
\]
where we used the strictly positive lower bound for $G$ as described in point $\textup{(i)}$. This latter estimate, in conjunction with the the sub-linear growth  of $x(t)$ will give a contradiction, hence the soliton-like solution cannot exist, and therefore the threshold for the scattering is given precisely by the quantity as in \eqref{sc-reg}.  \medskip\\
\noindent \textup{(iii-b)}   As for the blow-up in finite time, the last part of strategy can be considered similar, as it is given  by a Glassey argument based on virial identities. Nonetheless, the analysis is different and  more complicated with respect two points of the rigidity part for the scattering theorems. 
Specifically, in the formation of singularities scenario, we cannot rely on some compactness property on the nonlinear flow, hence the control on the remainder $H_R(u(t))$ cannot be given in a full generality. This is why we have to assume some symmetry hypothesis on the solution. It is here that we need to introduce in the framework $\Sigma_3$, the space of cylindrical symmetric solutions, with finite variance only along the third axis direction. 
Let us recall that even for the classical cubic NLS, (i.e. $\lambda_1=-1$ and $\lambda_2=0$), it is an open problem to show blow-up without assuming any additional symmetry hypothesis of finiteness of the variance, see \cite{AkN, HR, Gla, KRRT, HR07, HR2010, DWZ, OT, Guevara, Kav}.  Here we give the minimal assumptions to obtain formation of singularities in finite time, i.e. the solution is in $\Sigma_3$. See also \cite{Mar} for an early work on NLS in anisotropic spaces, and \cite{BFG-21, DF-ZAMP, ADF} for these techniques applied to other dispersive models. \medskip

For the virial argument, we chose here a (rescaled) function $\chi_R$ as the sum of a rescaled localization function  $\rho_R$, plus the function $x_3^2$. Here  $\rho_R$ is a well-constructed  function depending only on the two variables  $\bar x=(x_1,x_2)$ which provides a localization in the exterior of a cylinder, parallel to the $x_3$ axis and with radius of size $|\bar x|\sim R.$   The notation $|\bar x|$ clearly stands for $|\bar x|:=(x_1^2+x_2^2)^{1/2}.$
Moreover we added the not-localized  function $x_3^2$ in order to obtain a virial-like estimate of the form
\begin{equation*}\label{eq:intro-vir-bis}
\frac{d^2}{dt^2}V_{\rho_R+x_3^2}(t)\leq 4\int|\nabla u(t)|^2 dx +6 \lambda_1\int|u(t)|^4 dx+H_R(u(t)),
\end{equation*}
where the error  $H_R$ is defined by
\[
\begin{aligned}
H_R(u(t))&=4\lambda_1 \int a_R(\bar x)|u(t)|^4\, dx+cR^{-2}\\
&+2\lambda_2\int\nabla\rho_R\cdot\nabla\left(K\ast|u(t)|^2\right)|u(t)|^2\,dx-4\lambda_2\int x_3\partial_{x_3}\left(K\ast|u(t)|^2\right)|u(t)|^2\,dx,
%&=o_R(1)\|u(t)\|_{\dot H^1}^2 +\Lambda,
\end{aligned}
\]
and $a_R(\bar x)$ is a bounded, nonnegative function supported in the exterior of a cylinder of  radius of order $R$.  We estimate  $\int a_R(\bar x)|u|^4\, dx= o_R(1)\|u(t)\|_{\dot H^1}^2$ by means of a suitable Strauss embedding. Hence it remains to estimate the non-local terms in $H_R(u(t))$.
Similarly to the scattering part, the strategy is to split $u(t,x)$ by separating it in the interior and in the exterior of a cylinder, instead of a ball,  and computing the interaction given by the dipolar term.  The  further difficulty (with respect to the virial argument for the scattering theorem) is that $K\ast \cdot$ is not supported inside any cylinder, even if we localized the function where $K$ is acting on through the convolution.  Therefore, by performing further suitable splittings, we are able to give the identity
\[
\Lambda=6\lambda_2\int (K\ast|u(t)|^2)|u(t)|^2\,dx+ \epsilon_R,
\]
where the contributes defining $\epsilon_R$ consist of  terms of the form $\langle \mathcal R_3^4f,g\rangle$ when $f$ is supported in $\{|\bar x|\geq \gamma_2R\}$ while $g$ is supported in $\{|\bar x|\leq \gamma_1R\},$ for some positive parameters $\gamma_1$ and $\gamma_2$ satisfying $d:=\gamma_2-\gamma_1>0.$ Clearly, the localizations of $u(t,x)$ play the role of $f,g$ above. Hence, by means of  \autoref{lemma:decay-r4}  we can conclude, provided $R$ is large enough, with 
\[
\frac{d^2}{dt^2}V_{\rho_R+x_3^2}(t)\leq 4G(u(t))+\epsilon_{R}\leq -2\delta,
\] 
which in turn implies the finite time blow-up via a Glassey convexity argument \cite{Gla}. Note that we used the strictly negative upper bound for $G$ as described in point $\textup{(i)}$. 

\section{Sketch of the proofs below the threshold}

\subsection{Scattering}\label{subsec:sca} As already mentioned in point $\textup{(iii-a)}$, the scattering result given in \autoref{theo-scat-BF} is given by running a concentration/compactness and rigidity scheme, as pioneered by Kenig and Merle in their celebrated works \cite{KM1,KM2}. Nowadays there is a huge  literature  on this method, applied to several dispersive models, and since the scope of this review paper is not to go over the details of these techniques, we refer the reader to  \cite{AkN, CFX, FV, Guevara, DHR,HR} for mass-energy intracritical NLS equations. Let us only mention that the method can be viewed as an induction of the energy method, and it proceeds by contradiction, by assuming the the threshold for global and scattering solutions is strictly smaller than the claimed one. Hence we define the threshold for scattering as follows:
\begin{align*}
\mathcal{ME}=\sup&\left\{ \delta : \right. M(u_0)E(u_0)<\delta \hbox{ and } \|u_0\|_{L^2}\|\nabla u_0\|< \|Q\|_{L^2}\|\nabla Q\|
_{L^2} \\\notag
&\quad \hbox{then the solution to \eqref{GP} with initial data } u_0 \textit{ is in } L^8L^4\left.\right\}.
\end{align*}
A classical small data theory gives that if the initial datum is small enough in the energy norm, then the corresponding solution scatters, or equivalently it belongs to $L^8L^4:=L^8_t(\R; L^4_x(\R^3))$. Therefore the threshold is certainly strictly  positive. The goal is therefore to prove that $\mathcal{ME}=M(Q)E(Q)$. \medskip
At this point we assume by contradiction that the threshold is strictly smaller than the given one (i.e. we assume $\mathcal{ME}<M(Q)E(Q)$  and we  eventually prove that the latter leads to a contradiction). 

Indeed,  a linear profile decomposition theorem tailored for the equation \eqref{GP}, see \cite[Theorem 4.1, Proposition  4.3, and Corollary 4.4]{BF19}, and the existence of the wave operator, enable us to establish the following. 
\begin{theorem}\label{lemcri}
There exists a not trivial initial profile $u_{crit}(0)\in H^1$
with $M(u_{crit}(0))E(u_{crit}(0))=\mathcal{ME}$ and $ \|u_{crit}(0)\|_{L^2}\|\nabla u_{crit}(0)\|< \|Q\|_{L^2}\|\nabla Q\|
_{L^2} $ such that  the corresponding solution $u_{crit}(t)$ to \eqref{GP} is globally defined and does not scatter. Moreover, there exists a continuous function $x(t) : \R^+\mapsto\R^3$ such that $\{u_{crit}(t,x+x(t)),\,t\in\R^+\}$ is precompact as a subset of $H^1.$ Such a function $x(t)$ satisfies $|x(t)|=o(t)$ as $t\to+\infty$, namely it grows sub-linearly  at infinity.
\end{theorem}
The Theorem above  says that by assuming $\mathcal{ME}<M(Q)E(Q)$, we are able to construct an initial datum whose non-linear evolution is global and non-scattering. The precompactness  tells us that $u_{crit}(t)$ remains spatially localized (uniformly in time) along the continuous path $x(t)\in\R^3$. Specifically, for any $\varepsilon>0$ there exists $R_{\varepsilon}\gg1$ such that 
\begin{equation}\label{small-prec}
\int_{|x-x(t)|\geq R_{\varepsilon}}|\nabla u_{crit}(t)|^2+|u_{crit}(t)|^2+|u_{crit}(t)|^4 \leq\varepsilon \qquad \hbox{for any }\quad t\in\R^+.
\end{equation}
The proof of the growth property of $x(t)$ is inspired by \cite{DHR}, and it is based on  Galilean transformations of the solution.

The conclusion of the Kenig-Merle scheme is reached provided  we can  show that the solution given in \autoref{lemcri} cannot exist. Indeed, as introduced in  \autoref{subsec:vir} (iii-a), a virial argument will give, see \eqref{vir-3} 
\begin{equation*}\label{vir-31}
\begin{aligned}
\frac{d^2}{dt^2}V_{\chi_R}(t)&=4\int |\nabla u(t)|^2\,dx+6\lambda_1\int |u(t)|^4\,dx +\varepsilon_{1,R}\\
&-2\lambda_2R\int\nabla\chi\left(\frac xR\right)\cdot\nabla\left(K\ast|u(t)|^2\right)|u(t)|^2\,dx. 
\end{aligned}
\end{equation*}
The main goal is therefore to estimate  the non-local contribution $\Lambda:=-2\lambda_2R\int \nabla\chi\left(\frac xR\right)\cdot\nabla\left(K\ast|u(t)|^2\right)|u(t)|^2\,dx$. 
We introduce the  space localization inside and outside a ball of radius $10R$, namely we write (we ignore the time dependence)
\[u=\bold 1_{\{| x|\leq 10R\}}u+\bold 1_{\{|x|\geq 10R\}}u:=u_i+u_o.
\] 
By using the disjointness of the supports, we can rewrite $\Lambda=\Lambda_{i,i}+\Lambda_{o,i}$. In the latter notation, the subscript $\Lambda_{\diamond, \star}$, for $\diamond, \star\in\{i,o\}$, means the following: after quite lengthy and tricky manipulations, the terms we are considering are of the form:
\[
\Lambda_{\diamond, \star}(u)=\int g\left(K\ast|u_\diamond|^2\right)h(|u_\star|^2)\,dx,
\] 
i.e. the dipolar kernel $K$ acts (via the convolution) on the localisation (of $u$) given by the first symbol $\diamond$, while the other term in the integral contains the term localized according to the symbol $\star$. With a careful handling of the expression above, we reduce everything to fulfil the hypothesis of \autoref{lemma:in-out} and \autoref{lemma:out-in}, leading to the final estimate
\[
\Lambda\geq 6\lambda_2\int (K\ast|u(t)|^2)|u(t)|^2\,dx + \epsilon_{2,R}
\]
with
		\begin{equation}
		\begin{aligned} \label{est-A-R}
		\varepsilon_{2,R}&\lesssim  R^{-1} +R^{-1}\|u(t)\|^2_{H^1} \|u(t)\|^2_{L^4(|x| \geq10R)} + R^{-1} \|u(t)\|^2_{H^1} \\
		&+\|u(t)\|^2_{L^4(|x| \geq 10R)} +\|u(t)\|^4_{L^4(|x| \geq10R)}.
\end{aligned}
\end{equation}

\noindent Let us observe that the remainder as in \eqref{est-A-R} has a similar form as the one in \eqref{rem-eps1} describing $\varepsilon_{1,R}$. Hence they can be controlled in the same fashion. Specifically, we fix a time interval $[T_0,T_1]$  for $0<T_0<T_1$ and we take $R\geq \sup_{[T_0,T_1]}|x(t)|+R_{\varepsilon}$ as in \eqref{small-prec} such that $\frac{d^2}{dt^2}z_R(t)\geq \frac\alpha2>0$. An integration on $[T_0,T_1]$ yields to 
\begin{equation*}\label{exti-1}
R\gtrsim R\|u\|_{L^2}\|\nabla u\|_{L^2}\gtrsim\left|\frac{d}{dt}z_R(T_1)-\frac{d}{dt}z_R(T_0)\right|\geq \frac\alpha2(T_1-T_0),
\end{equation*}
i.e. for come $c>0$, we have $c(T_1-T_0)\leq R$.  Note that by the sub-linearity growth of $x(t)$, once fixed  $\delta>0$, we can guarantee that there exists a time $t_{\delta}$ such that $|x(t)|\leq \delta t$ for any $t\geq t_{\delta}$. Hence, by picking $\delta=c/2$, and $R=R_\varepsilon+\frac{cT_1}{2}$, we have $\frac{cT_1}{2}\leq R_\varepsilon+cT_0$, and the latter leads a contradiction, as we can let $T_1$ be as large as we want, while the right hand-side remains bounded. 

\subsection{Blow-up} As introduced in the \autoref{subsec:vir}, our goal is to give the following estimate:
\begin{equation}\label{eq:intro-vir2}
\frac{d^2}{dt^2}V_{\rho_R+x_3^2}(t)\leq 4\int|\nabla u(t)|^2 dx +6 \lambda_1\int |u(t)|^4 dx+H_R(u(t)),
\end{equation}
where the error  $H_R$ is defined by
\begin{equation}\label{eq:rema}
\begin{aligned}
H_R(u(t))&=4\lambda_1 \int a_R(\bar x)|u(t)|^4\, dx+cR^{-2}\\
&+2\lambda_2\int \nabla\rho_R\cdot\nabla\left(K\ast|u(t)|^2\right)|u(t)|^2\,dx-4\lambda_2\int x_3\partial_{x_3}\left(K\ast|u(t)|^2\right)|u(t)|^2\,dx 
\end{aligned}
\end{equation}
and $a_R=0$ in $\{|\bar x|\leq R\}$. To this aim, we consider a regular, nonnegative, radial function $\rho=\rho(|\bar x|)=\rho(r)$ such that
\[
\rho(r)=
\begin{cases}
r^2 &\hbox{ if }\, r\leq1\\
0  &\hbox{ if }\, r\geq2
\end{cases},
\quad \hbox{such that} \quad \rho^{\prime\prime}\leq2 \quad\hbox{for any} \quad r\geq0.
\] 
A similar function can be explicitly constructed, see \cite{Mar, BF21, OT}, and satisfies \eqref{eq:intro-vir2}, \eqref{eq:rema}, with $a_R$ localized in the exterior of a cylinder of radius $R$.  By means of Strauss estimates, it is quite easy to obtain
\begin{align*}
H_R(u(t))&=2\lambda_2 \left(\int \nabla\rho_R\cdot\nabla\left(K\ast|u(t)|^2\right)|u(t)|^2\,dx-2\int x_3\partial_{x_3}\left(K\ast|u(t)|^2\right)|u(t)|^2\,dx \right)\\
&+o_R(1)\|u(t)\|_{\dot H^1}^2:=2\lambda_2( \Xi+ \Upsilon)+o_R(1)\|u(t)\|_{\dot H^1}^2.
\end{align*}
So we are reduced to the estimate of $\Xi+\Upsilon.$ In order to use the decays as in \autoref{sec-riesz}, we proceed with several localizations, in order to reduce the problems given by the non-local terms $\Xi+\Upsilon$ to fulfil the hypothesis of the decay properties for powers of the Riesz transforms. The scheme is as follows. We introduce the first localization inside and outside a cylinder of radius $10R$, namely we write (we ignore the time dependence)
\[u=\bold 1_{\{|\bar x|\leq 10R\}}u+\bold 1_{\{|\bar x|\geq 10R\}}u:=u_i+u_o.
\] 
By using the disjointness of the supports, we can rewrite $\Xi+\Upsilon=\Xi_{o,i}+\Xi_{i,i}+\Upsilon.$ The proof of the decay for $\Xi_{o,i}$ can be given, after careful manipulations, by means of the point-wise decay as in  \autoref{lemma:in-out} and \autoref{lemma:out-in}. The main problem is given by the term $\Xi_{i,i}+\Upsilon$ where we do not have any localization at the exterior of a cylinder, preventing us to obtain some decay straightforwardly. Hence a further splitting is introduced. 
We separate $u_i$ as $u_i=w_{i,i}+w_{i,o},$ where 
\[
w_{i,i}=\bold 1_{\{|\bar x|\leq R/10\}}u_i \quad \hbox{ and } \quad w_{i,o}=\bold 1_{\{|\bar x|\geq R/10\}}u_i=\bold 1_{\{R/10\leq|\bar x|\leq 4R\}}u.
\]
Therefore, we generate terms localized outside a cylinder  of radius $\sim R$,  specifically of the form 
\begin{align*}
\mathcal A_{i, o}(u_i)&=\int g\left(K\ast|w_{i,i}|^2\right)h(|w_{i,o}|^2)\,dx,\\
\mathcal B_{o, o}(u_i)&=\int \tilde g\left(K\ast|w_{i,o}|^2\right)\tilde h(|w_{i,o}|^2)\,dx,
\end{align*}
plus a quantity 
\[
\mathcal C_{i,i}(u)+\Upsilon:=2\int\bar x\cdot\nabla_{\bar x}\left(K\ast|u_i|^2\right)|u_i|^2\,dx+2\int x_3\partial_{x_3}\left(K\ast |u|^2\right)|u|^2\,dx.
\]
By continuing the (quite involved) computations, we end up with a reduction of $\mathcal A_{i, o}(u_i)$ and $\mathcal B_{o, o}(u_i)$ to a  framework as in \autoref{lemma:in-out} and \autoref{lemma:out-in}, and we get 
\begin{equation*}
\mathcal A_{i, o}(u_i)= o_R(1)\|u\|_{\dot H^1}^2\quad \hbox{ and } \quad \mathcal B_{o, o}(u_i)= o_R(1)\|u\|_{\dot H^1}^2.
\end{equation*}

\noindent In order to control the remainder term  $\mathcal C_{i,i}(u)+\Upsilon$ and  to make appear    $6\lambda_2\int (K\ast|u|^2)|u|^2\,dx$ in \eqref{eq:intro-vir2} that will yield to the whole quantity $4G(u(t)),$ we need to use the identity $2\int x\cdot\nabla\left(K\ast f\right)f\,dx=-3\int\left(K\ast f\right)f\,dx$. The latter follows from the relation $\xi\cdot \nabla_\xi\hat K=0.$   By observing that 
\[
\begin{aligned}
\mathcal C_{i,i}(u)+\Upsilon&=3\int \left(K\ast|u_i|^2\right)|u_i|^2\,dx-2\int x_3\partial_{x_3}\left(K\ast|u_i|^2\right)|u_o|^2\,dx\\
&-2\int x_3\partial_{x_3}\left(K\ast|u_o|^2\right)|u_i|^2\,dx-2\int x_3\partial_{x_3}\left(K\ast|u_o|^2\right)|u_o|^2\,dx
\end{aligned}
\]
and that 
\[
\xi_3\partial_{\xi_3}\hat K=8\pi\frac{\xi_3^2(\xi_1^2+\xi_2^2)}{|\xi|^4}=8\pi\left(\frac{\xi_3^2}{|\xi|^2}-\frac{\xi_3^4}{|\xi|^4}\right)=8\pi \widehat{\mathcal{R}_3^2}-8\pi\widehat{\mathcal{R}_3^4},
\]
we reduce the problem to the estimate of $\langle \mathcal R_3^4f,g\rangle_{L^2}$ when $f$ is supported in $\{|\bar x|\geq \gamma_2R\}$ while $g$ is supported in $\{|\bar x|\leq \gamma_1R\},$ for some positive parameters $\gamma_1$ and $\gamma_2$ satisfying $d:=\gamma_2-\gamma_1>0.$ Note that in the latter identity we used the fact that  $\frac{\xi_3^2}{|\xi|^2}$ and $\frac{\xi_3^4}{|\xi|^4}$ are (up to constants) the symbols, in Fourier space, of the operators $\mathcal{R}_3^2$ and $\mathcal{R}_3^4$, respectively. $\mathcal R_j^4$ denotes the fourth power of the Riesz transform, and $\widehat{\mathcal{R}_j^4}$ its symbol in Fourier space. 
At this point we use the estimate of \autoref{lemma:decay-r4} (for the contribution involving  $\mathcal{R}_j^4$ ), and again \autoref{lemma:in-out} and \autoref{lemma:out-in} (for the contribution involving $\mathcal{R}_j^2$). 
Thus, by summing up together the  estimates  we have 
\begin{equation*}\label{eq:vir3}
\frac{d^2}{dt^2}V_{\rho_R+x_3^2}(t)\leq 4G(t)+o_R(1)\|u(t)\|_{\dot H^1}^2\lesssim -1,
\end{equation*}
which allows to close a Glassey-type  convexity argument.

\subsection{Grow-up}
We give now the proof of the grow-up result, by sketching the proof of  \autoref{theo-blow-crite}. The proof follows the approach by Du, Wu, Zhang, see \cite{DWZ} (see also the results by Holmer and Roudenko \cite{HR2010}). It is done by contradiction, and it makes use of the so-called almost finite propagation speed, which enables us to control the quantity  $\|u(t)\|_{L^2(|x| \gtrsim R)}$ for sufficiently large times. It is well-known that, contrary to the wave equation, the Schr\"odinger equation doesn't enjoy a finite propagation speed; nonetheless, we can claim the following: provided $\sup_{t\in [0,\infty)} \| u(t)\|_{\dot H^1} <\infty$, then for any $\eta>0$, there exists a constant $C>0$ independent of $R$ such that for any $t\in [0,T]$ with $T:= \frac{\eta R}{C}$,
\begin{equation} \label{est-L2-norm}
\int_{|x|\gtrsim R} |u(t,x)|^2 dx \leq \eta + o_R(1).
\end{equation}
Indeed, let $\vartheta$ be a smooth radial function satisfying
\[
\vartheta (x) = \vartheta(r) = \left\{
\begin{array}{ccc}
0 &\text{if}& r \leq \frac{c}{2}, \\
1 &\text{if}& r \geq c,
\end{array}
\right.
\quad \vartheta'(r) \leq 1 \text{ for any } r\geq 0,
\]
where $c>0$ is a given constant. For $R>1$, we denote the radial function $\psi_R(x) = \psi_R(r) := \vartheta(r/R)$. We plug this function in the virial quantity $V_{\chi}$ defined in \eqref{vir-1}, and by the fundamental theorem of calculus, and by assuming that $\sup_{t\in [0,\infty)} \| u(t)\|_{\dot H^1} <\infty$, we have
	\[
	V_{\psi_R}(t) = V_{\psi_R}(0) + \int_0^t V^\prime_{\psi_R}(s) ds \leq V_{\psi_R}(0) + t\sup_{s\in [0,t]} |V^\prime _{\psi_R}(s)| \leq V_{\psi_R}(0) + CR^{-1} t.
	\]
By the choice of $\vartheta$, we have $V_{\psi_R}(0) = o_R(1)$ as $R\to\infty$. 
Since $V_{\psi_R}(t)\geq \int_{|x| \geq cR} |u(t,x)|^2 dx $, we obtain the  control on $L^2$-norm of the solution outside a large ball as in \eqref{est-L2-norm}.  By repeating the estimates as in \autoref{subsec:sca},  and by means of the  Gagliardo-Nirenberg interpolation inequality applied to \eqref{est-A-R}, we see that
	\begin{align} \label{est-z-R-appl}
	V^{\prime\prime}_{\varphi_R}(t) \lesssim  G(u(t)) +\left( R^{-1} + \|u(t)\|^{1/2}_{L^2(|x| \gtrsim R)}+\|u(t)\|_{L^2(|x| \gtrsim R)}\right).
	\end{align}
 
	Combining \eqref{est-L2-norm} and \eqref{est-z-R-appl}, we obtain that 
	for any $\eta \in (0,1)$, there exists a constant $C>0$ independent of $R$ such that for any $t\in [0,T]$ with $T:= \frac{\eta R}{C}$ such that
	\[
	V^{\prime\prime}_{\varphi_R}(t) \lesssim  G(u(t)) + \left(\left(\eta + o_R(1)\right)^{1/4}+\left(\eta + o_R(1)\right)^{1/2}\right)
	\]
	By the assumption \eqref{blow-crite}, we choose $\eta>0$ sufficiently small and $R>1$ sufficiently large to have
$V^{\prime\prime}_{\varphi_R}(t) \lesssim -\delta <0$
	for all $t\in [0,T]$. If we integrate  in time twice from $0$ to $T$, we get $ V^{\prime\prime}_{\varphi_R}(T) \leq o_R(1) R^2 - \frac{\delta \eta^2}{2C^2} R^2,$
	and by choosing  $R$ large enough, we obtain $z_{\varphi_R}(T) \leq -\frac{\delta \eta^2}{4C}R^2 <0,$ a contradiction with respect to the fact that $z_{\varphi_R}(T)$ is a nonnegative quantity.

\section{Sketch of the proofs above the threshold}
The dynamics above the threshold is a consequence of the following general theorem, where a sufficient condition to have global existence and scattering is given. It will be used to establish the asymptotic dynamics when the initial datum lies at the threshold as well (see later on, specifically see the proofs in Section 5). 	
	\begin{theorem} \label{theo-scat-crite}
		Let $\lambda_1$ and $\lambda_2$ satisfy \eqref{cond-GW}. Let $Q$ be a ground state related to \eqref{ell-equ}. Let $u(t)$ be a $H^1$-solution to \eqref{GP} defined on the maximal forward time interval $[0,T_{max})$. Assume that
		\begin{align} \label{scat-crite}
		\sup_{t\in [0,T_{max})} -P(u(t)) M(u(t)) < -P(Q) M(Q).
		\end{align}
		Then $T_{max}=\infty$ and the solution $u(t)$ scatters in $H^1$ forward in time.
	\end{theorem}\medskip
The proof of the Theorem above is done by employing a concentration/compactness and rigidity road map, as for the case below the threshold, see \autoref{theo-scat-BF}. As mentioned in the paragraph before \autoref{lemcri}, the main tool to prove existence of global and non-scattering solution is given by a profile decomposition theorem, which is a linear statement; so, in order to construct non-linear profiles, the existence of wave operator is used. Moreover, when we are in the case below the threshold, such a non-linear profiles can be proved to be global and scattering. When we do not assume initial data below the threshold, such a control on the non-linear profiles cannot be given. Nonetheless, we are able to prove a non-linear profile decomposition theorem along bounded non-linear flows, which overcome the lack of finiteness of the scattering norm of the non-linear profiles. See \cite[Lemma 3.1]{DFH}. The latter result in \cite{DFH} was inspired to \cite{Guevara}, where the NLS case was treated. We recall also here (as we remarked it in the Introduction) that the restriction to the region \eqref{cond-GW} is imposed to guarantee the negative sign of the potential energy, which is fundamental to get the right bounds on the non-linear profiles constructed when running a Kenig-Merle scheme. \\

\textit{Proof of  \autoref{theo-scat-above}}. Let $u_0 \in \Sigma$ satisfying  all the conditions in \eqref{above:sca}. We will show that \eqref{scat-crite} holds true, which in turn  implies the result, by means of  \autoref{theo-scat-crite}. The strategy is  in the spirit of Duyckaerts and Roudenko \cite{DR2}, and it is done in three steps, and it is based on an ODE argument.  We summarize the main steps by just explaining how the method works, and by defining the basic objects. For a comprehensive proof we refer the reader to \cite{DFH}, where all the details are given. \\

By easy computations, we have 	
\begin{align} \label{iden-N-H}
	-P(u(t)) = 4E(u) - V''(t), \quad H(u(t)) = 6E(u) - V''(t),
\end{align}
and by using that $P(u(t))$ is negative (recall that we are working in $RUR$), then $V''(t) \leq 4E(u)$. At this point we recall, see \cite{GW}, that for any $ f\in \Sigma$
\begin{align} \label{est-GW}
\left(\Im \int  x\cdot \nabla f(x) \overline{f}(x)dx \right)^2 \leq \|x f\|^2_{L^2} \left( H(f) - \frac{(-P(f))^{\frac{2}{3}}}{(C_{GN})^{\frac{2}{3}} (M(f))^{\frac{1}{3}}} \right). 
\end{align}
By plugging  \eqref{iden-N-H} into \eqref{est-GW}, we have 
\begin{align*}
\left(\frac{V'(t)}{2}\right)^2 \leq V(t) \left[ 6E(u)-V''(t) - \frac{(4E(u)-V''(t))^{\frac{2}{3}}}{(C_{GN})^{\frac{2}{3}} (M(u))^{\frac{1}{3}}} \right]
\end{align*}
We introduce the function $z(t):= \sqrt{V(t)}$,  and  we define 
$ h(\zeta):= 6E(u)-\zeta-\frac{(4E(u)-\zeta)^{\frac{2}{3}}}{(C_{GN})^{\frac{2}{3}} (M(u))^{\frac{1}{3}}}$ for $\zeta \leq 4E(u)$. 
We can now rewrite the estimate above as  $(z'(t))^2 \leq h(V''(t))$.	The function $h(\zeta)$ on the unbounded interval $(-\infty, 4E(u))$ has   a minimum in  $\zeta_0$ defined through  
\begin{equation*} \label{defi-lambda-0}
	1= \frac{2(4E(u)-\zeta_0)^{-\frac{1}{3}}}{3 (C_{GN})^{\frac{2}{3}} (M(u))^{\frac{1}{3}}},
\end{equation*}
and in particular $h(\zeta_0) = \zeta_0/2$. The precise expression for $C_{GN}$ given in \eqref{inde-quant-proof}, yields to 
\begin{align} \label{iden-lambda-0}
\frac{E(u) M(u)}{E(Q) M(Q)} \left(1-\frac{\zeta_0}{4E(u)}\right) =1.
\end{align}

\noindent \textup{(i)} By using the previous relations, the first point of the ODE argument consists in rewriting the scattering conditions in \eqref{above:sca} in an alternative way, by using the functions $z(t)$, $V(t)$, $h(\zeta)$, and the value $\zeta_0$.  
From the hypothesis that $M(u_0)E(u_0)\geq M(Q)E(Q)$, we get  that  \eqref{iden-lambda-0} is equivalent to $\zeta_0 \geq 0.$  The  second condition in \eqref{above:sca} is equivalent to  
\begin{equation*} \label{cond-2-above-equi}
	(z'(0))^2 \geq \frac{\zeta_0}{2} = h(\zeta_0),
\end{equation*}
while   the third condition in \eqref{above:sca} is equivalent to $V''(0) >\zeta_0$. The last condition in \eqref{above:sca} is instead equivalent to $z'(0) \geq 0$.\medskip

\noindent\textup{(ii)} The previous conditions replacing the ones in \eqref{above:sca}, jointly with a continuity argument, yield to a lower bound  
\begin{equation}\label{claim-above}
V''(t)\geq  \zeta_0 + \delta_0,
\end{equation} 
for some $\delta_0>0$ and for any $t\in[0, T_{max})$. \medskip

\noindent\textup{(iii)}  Eventually, we are able to prove \eqref{scat-crite}. It follows from \eqref{claim-above} and by using  that $\zeta_0\geq0$, \eqref{iden-lambda-0}, and \eqref{inde-quant-proof},  that
	\begin{align*}
	-P(u(t)) M(u(t)) &= (4 E(u) - V''(t)) M(u) \leq (4E(u) -\zeta_0 -\delta_0) M(u) \\
	&\leq 4 E(Q) M(Q) - \delta_0 M(u)  = -(1-\eta) P(Q) M(Q) 
	\end{align*}
	for all $t\in [0,T_{max})$, where $\eta:= \frac{\delta_0M(u)}{4E(Q) M(Q)}>0$. This shows \eqref{scat-crite} and we can conclude the proof of  \autoref{theo-scat-above}.

\section{Sketch of the proofs at the threshold}

We now consider the threshold case, i.e. when the initial data  satisfy \eqref{cond-ener-at}, and we give an overview on the proof of \autoref{theo-dyna-at}. Firstly, let us observe that in \eqref{cond-ener-at} we can assume, by scaling invariance, that $M(u_0) = M(Q)$ and  $E(u_0) = E(Q)$. We continue with the proof of the three points in order. \medskip

\noindent \textup{(i)} As we are considering $M(u_0) = M(Q)$ and  $E(u_0) = E(Q)$, we see that  \eqref{cond-scat-at} becomes $H(u_0)<H(Q)$. Then in order to prove that \eqref{cond-ener-at} and \eqref{cond-scat-at} imply that the solution is global, it is enough  to prove that the kinetic energy remains bounded by $H(Q)$ (by the blow-up alternative). By the absurd, if we assume that there exists a time $\tau$ in the lifespan of the solution such that  $H(u(\tau))=H(Q)$, then we obtain by definition of the energy, that $- P(u(\tau)) =H(u(\tau))- 2E(u(\tau)) = H(Q) - 2 E(Q) =- P(Q).$ Namely  $u(\tau)$ is an optimizer of \eqref{GN-ineq}. A  Lions' concentration-compactness type-lemma, see \cite[Lemma 5.1]{DFH}, implies  that $u(t)$ is a (rescaling) of a ground state related to \eqref{ell-equ} multiplied by a (time dependent) phase shift. This yields to a contradiction with respect to the hypothesis, as we would have $H(u_0) M(u_0)=H(Q) M(Q)$;   therefore, by the blow-up alternative, $u(t)$ is globally defined. \medskip

Under the hypothesis that the coefficients $\lambda_1$ and $\lambda_2$ satisfy \eqref{cond-GW}, then we are able to prove that we have the result in the second part of \autoref{theo-dyna-at} \textup{(i)}, by distinguishing two cases.  \medskip

\noindent We firstly suppose that $\sup_{t\in [0,\infty)} H(u(t)) < H(Q)$. This means that  there exists $\varepsilon>0$ such that for all $t\in [0,\infty)$ (the solutions is global), $	H(u(t)) \leq (1-\varepsilon) H(Q).$ By plugging the best constant (given in term of the ground state to \eqref{ell-equ}) of the Gagliardo-Sobolev type estimate {\color{red} CITE}, it is straightforward to see that 
\[
-P(u(t)) M(u(t)) \leq C_{GN} \left(H(u(t)) M(u(t))\right)^{\frac{3}{2}} \leq-(1-\varepsilon)^{\frac{3}{2}} P(Q) M(Q) 
\]
hence the condition \eqref{scat-crite} of  \autoref{theo-scat-crite} holds true, and the solution scatters forward in time.  \medskip

\noindent  If instead $\sup_{t\in[0,\infty)} H(u(t)) = H(Q)$, then there exists a time sequence $(t_n)_{n \geq 1} \subset [0,\infty)$ such that
	\[
	M(u(t_n)) = M(Q), \quad E(u(t_n))= E(Q), \quad \lim_{n\rightarrow \infty} H(u(t_n)) = H(Q).
	\]
Moreover, $t_n\to \infty$. Indeed, if (up to subsequences) $t_n \rightarrow \tau$, as $u(t_n) \rightarrow u(\tau)$ strongly in $H^1$, then it can be shown that $u(\tau)$ is an optimizer for \eqref{GN-ineq}. Arguing as above, we have a contradiction. A Lions-type lemma \cite[Lemma 5.1]{DFH} gives the desired result.\medskip

\noindent \textup{(ii)} We continue with the proof of the second point. Suppose the initial datum satisfies \eqref{cond-ener-at} and \eqref{cond-at}. By scaling, we reduce to the case $
	M(u_0) = M(Q)$, $ E(u_0) = E(Q)$, hence $ H(u_0) = H(Q)$. Hence $u_0$ is an optimizer for \eqref{GN-ineq}. This shows that $u_0(x) = e^{i\theta} \mu \tilde{Q}(\mu x)$ for some $\theta \in \R,$  $\mu>0$ and $\tilde Q$ a ground state related to \eqref{ell-equ}. By the uniqueness of solutions, we end-up with $u(t,x) = e^{i\mu^2t} e^{i\tilde{\theta}} \mu \tilde{Q}(\mu x)$ for some $\tilde{\theta}\in \R$.\medskip

\noindent	\textup{(iii)} Finally, suppose that $u_0 \in H^1$ satisfies \eqref{cond-ener-at} and \eqref{cond-blow-at}. By scaling we have reduced \eqref{cond-blow-at} to $H(u_0)> H(Q)$. By the same argument in the proof of the first point, we claim that $H(u(t))> H(Q)$, for every time in the lifespan of the solution.  If the maximal time of existence is finite there is nothing to prove. Otherwise, if the solution exists for all times, we separate the analysis in two   cases.\medskip
	
\noindent Suppose $\sup_{t\in [0,\infty)} H(u(t)) > H(Q)$.  Hence there exists $\varepsilon>0$ such that for all $t\in [0,\infty)$,
$H(u(t)) \geq (1+\eta) H(Q)$. By using the definition  \eqref{def:G} of $G$ and the previous property, we have  \[
	G(u(t)) M(u(t)) \leq 3 E(Q) M(Q) - \frac{1}{2} (1+\eta) H(Q) M(Q) = -\frac{\eta}{2}H(Q) M(Q)<0,
	\]
	for all $t\in [0,\infty),$ where in the last equality we used \eqref{inde-quant-proof}. By applying  \autoref{theo-blow-crite}, we finish the proof.\medskip
	
\noindent If instead $\sup_{t\in [0,\infty)} H(u(t)) = H(Q)$, similarly to above we have that  there exist a diverging  sequence of times $\{t_n\}$, a ground state $\tilde{Q}$ related to \eqref{ell-equ}, and a sequence $\{y_n\}_{n\geq 1} \subset \R^3$ such that $ u(t_n, \cdot +y_n) \rightarrow e^{i\theta} \mu \tilde{Q}(\mu \cdot)$ in $H^1$, for some $\theta \in \R$ and $\mu>0$ as $n\rightarrow \infty$. This conclude the proof of \autoref{theo-dyna-at}.

\subsection*{Acknowledgements}
\noindent J.B. was partially supported by ``Problemi stazionari e di evoluzione nelle equazioni di campo nonlineari dispersive'' of GNAMPA 2020.
L.F. was supported by the EPSRC New Investigator Award (grant no. EP/S033157/1).

%%%%%%%%%%%%%%%%% REFERENCES


\begin{bibdiv}
\begin{biblist}

\bib{AkN}{article}{
   author={Akahori, Takafumi},
   author={Nawa, Hayato},
   title={Blowup and scattering problems for the nonlinear Schr\"{o}dinger
   equations},
   journal={Kyoto J. Math.},
   volume={53},
   date={2013},
   number={3},
   pages={629--672},
   issn={2156-2261},
   %review={\MR{3102564}},
   %doi={10.1215/21562261-2265914},
}
\bib{AEMWC}{article}{
   author={Anderson, M.H.},
   author={Ensher, J.R.},
    author={Matthews, M.R.},
   author={Wieman, C.E.},
   author={Cornell, E.A.},
   title={Observation of Bose-Einstein Condensation in a Dilute Atomic Vapor},
   journal={Science},
   volume={269},
   date={1995},
   number={5221},
   pages={198--201}, 
   %review={\MR{2748730}},
}

\bib{AS}{article}{
   author={Antonelli, Paolo},
   author={Sparber, Christof},
   title={Existence of solitary waves in dipolar quantum gases},
   journal={Phys. D},
   volume={240},
   date={2011},
   number={4-5},
   pages={426--431},
   issn={0167-2789},
   %review={\MR{2748730}},
}

\bib{ADF}{article}{
   author={Ardila, {A. H.}},
   author={Dinh, {V. D.}},
   author={Forcella, {L.}},
   title={Sharp conditions for scattering and blow-up for a system of NLS arising in optical materials with $\chi^3$ nonlinear response},
   journal={Communications in Partial Differential Equations},
   volume={46},
   date={2021},
  number={11},
   pages={2134--2170},
   %issn={0022-1236},
   %review={\MR{3985521}},
   %doi={10.1016/j.jfa.2019.04.005},
}


\bib{BaCa}{article}{
   author={Bao, W.},
   author={Cai, Y.},
   title={Mathematical Theory and Numerical methods for Bose-Einstein condensation},
   journal={Kinetic and Related Models AMS},
   volume={6},
   date={2013},
   number={1},
   pages={1--135},
   %review={\MR{2748730}},
}

\bib{BF19}{article}{
   author={Bellazzini, Jacopo},
   author={Forcella, Luigi},
   title={Asymptotic dynamic for dipolar quantum gases below the ground
   state energy threshold},
   journal={J. Funct. Anal.},
   volume={277},
   date={2019},
   number={6},
   pages={1958--1998},
   issn={0022-1236},
   %review={\MR{3985521}},
   %doi={10.1016/j.jfa.2019.04.005},
}

\bib{BF21}{article}{
   author={Bellazzini, {J.}},
   author={Forcella, {L.}},
   title={Dynamical collapse of cylindrical symmetric dipolar Bose-Einstein condensates},
   journal={Calc. Var.},
   volume={60},
   date={2021},
   number={229},
   pages={1--33},
   %issn={0022-1236},
   %review={\MR{3985521}},
   %doi={10.1016/j.jfa.2019.04.005},
}

\bib{BFG-21}{article}{
   author={Bellazzini, {J.}},
   author={Forcella, {L.}},
   author={Georgiev, {V.}},
   title={Ground state energy threshold and blow-up for NLS with competing nonlinearities},
   journal={arXiv:2012.10977 [math.AP]},
   %volume={277},
  % date={2019},
   %number={6},
   %pages={1958--1998},
   %issn={0022-1236},
   %review={\MR{3985521}},
   %doi={10.1016/j.jfa.2019.04.005},
}


\bib{BJ}{article}{
   author={Bellazzini, Jacopo},
   author={Jeanjean, Louis},
   title={On dipolar quantum gases in the Unstable Regime},
   journal={SIAM J. Math. Anal.},
   volume={48},
   date={2016},
   number={3},
   pages={2028--2058},
   issn={0036-1410},
   %review={\MR{3510005}},
  % doi={10.1137/15M1015959},
}	

\bib{BrSaToHu}{article}{
   author={Bradley, C.C.},
   author={Sackett, C. A.},
   author={Tolett, J.J.},
   author={Hulet, R.J.},
   title={Evidence of Bose-Einstein Condensation in an Atomic Gas with Attractive Interaction},
   journal={Phys. Rev. Lett. },
   volume={75},
   date={1995},
   pages={1687--1690},
   %review={\MR{3510005}},
  % doi={10.1137/15M1015959},
}		

\bib{CMS}{article}{
   author={Carles, R\'emi},
   author={Markowich, Peter A.},
   author={Sparber, Christof},
   title={On the Gross-Pitaevskii equation for trapped dipolar quantum
   gases},
   journal={Nonlinearity},
   volume={21},
   date={2008},
   number={11},
   pages={2569--2590},
   issn={0951-7715},
   %review={\MR{2448232}},
}

\bib{CFX}{article}{
   author={Fang, DaoYuan},
   author={Xie, Jian},
   author={Cazenave, Thierry},
   title={Scattering for the focusing energy-subcritical nonlinear
   Schr\"{o}dinger equation},
   journal={Sci. China Math.},
   volume={54},
   date={2011},
   number={10},
   pages={2037--2062},
   issn={1674-7283},
   %review={\MR{2838120}},
   %doi={10.1007/s11425-011-4283-9},
}

\bib{DMAVDKK}{article}{
   author={Davis, Kendall B.},
   author={Mewes, M.-O.},
   author={Andrews, Michael R.},
   author={Van Druten, N.J.},
   author={Durfee, D.S.},
   author={Kurn, D.M.},
   author={Ketterle, Wolfgang},
   title={Bose-Einstein condensation in a gas of sodium atoms},
   journal={Physical Review Letters},
   volume={75},
   date={1995},
   number={22},
   pages={3639}
   %review={\MR{3485846}},
   %doi={10.3934/dcds.2016.36.3639},
}



\bib{DFH}{article}{
   author={Dinh, Van Duong},
   author={Forcella, Luigi},
   author={Hajaiej, Hichem},
   title={Mass-energy threshold dynamics for dipolar Quantum Gases},
   journal={Communications in Mathematical Sciences, to appear. arXiv:2009.05933 [math.AP]},
   %volume={277},
  % date={2019},
   %number={6},
   %pages={1958--1998},
   %issn={0022-1236},
   %review={\MR{3985521}},
   %doi={10.1016/j.jfa.2019.04.005},
}
\bib{DF-ZAMP}{article}{
   author={Dinh, Van Duong},
   author={Forcella, Luigi},
   title={Blow-up results for systems of nonlinear Schr\"{o}dinger equations
   with quadratic interaction},
   journal={Z. Angew. Math. Phys.},
   volume={72},
   date={2021},
   number={5},
   pages={Paper No. 178},
   issn={0044-2275},
   %review={\MR{4309791}},
   %doi={10.1007/s00033-021-01607-6},
}
		

\bib{DWZ}{article}{
   author={Du, Dapeng},
   author={Wu, Yifei},
   author={Zhang, Kaijun},
   title={On blow-up criterion for the nonlinear Schr\"{o}dinger equation},
   journal={Discrete Contin. Dyn. Syst.},
   volume={36},
   date={2016},
   number={7},
   pages={3639--3650},
   issn={1078-0947},
   %review={\MR{3485846}},
   %doi={10.3934/dcds.2016.36.3639},
}
		
\bib{DHR}{article}{
   author={Duyckaerts, Thomas},
   author={Holmer, Justin},
   author={Roudenko, Svetlana},
   title={Scattering for the non-radial 3D cubic nonlinear Schr\"odinger
   equation},
   journal={Math. Res. Lett.},
   volume={15},
   date={2008},
   number={6},
   pages={1233--1250},
   issn={1073-2780},
   %review={\MR{2470397}},
   %doi={10.4310/MRL.2008.v15.n6.a13},
}
\bib{DR1}{article}{
   author={Duyckaerts, Thomas},
   author={Roudenko, Svetlana},
   title={Threshold solutions for the focusing 3D cubic Schr\"{o}dinger
   equation},
   journal={Rev. Mat. Iberoam.},
   volume={26},
   date={2010},
   number={1},
   pages={1--56},
   issn={0213-2230},
  % review={\MR{2662148}},
  % doi={10.4171/RMI/592},
}
\bib{DR2}{article}{
   author={Duyckaerts, {T.}},
   author={Roudenko, {S.}},
   title={Going beyond the threshold: scattering and blow-up in the focusing
   NLS equation},
   journal={Comm. Math. Phys.},
   volume={334},
   date={2015},
   number={3},
   pages={1573--1615},
   issn={0010-3616},
   %review={\MR{3312444}},
   %doi={10.1007/s00220-014-2202-y},
}

\bib{FGG}{article}{
   author={Ferrero, Alberto},
   author={Gazzola, Filippo},
   author={Grunau, Hans-Christoph},
   title={Decay and eventual local positivity for biharmonic parabolic
   equations},
   journal={Discrete Contin. Dyn. Syst.},
   volume={21},
   date={2008},
   number={4},
   pages={1129--1157},
   issn={1078-0947},
   %review={\MR{2399453}},
   %doi={10.3934/dcds.2008.21.1129},
}
\bib{FV}{article}{
   author={Forcella, Luigi},
   author={Visciglia, Nicola},
   title={Double scattering channels for 1D NLS in the energy space and its
   generalization to higher dimensions},
   journal={J. Differential Equations},
   volume={264},
   date={2018},
   number={2},
   pages={929--958},
   issn={0022-0396},
   % review={\MR{3720834}},
  % doi={10.1016/j.jde.2017.09.027},
}
\bib{GW}{article}{
   author={Gao, Y.},
   author={Wang, Z.},
   title={Blow-up for trapped dipolar quantum gases with large energy},
   journal={J. Math. Phys.},
   volume={60},
   date={2019},
   number={12},
   pages={121501, 10},
   issn={0022-2488},
   %review={\MR{4037401}},
   %doi={10.1063/1.5121793},
}

		

\bib{Gla}{article}{
   author={Glassey, R. T.},
   title={On the blowing up of solutions to the Cauchy problem for nonlinear
   Schr\"odinger equations},
   journal={J. Math. Phys.},
   volume={18},
   date={1977},
   number={9},
   pages={1794--1797},
   issn={0022-2488},
   %review={\MR{0460850}},
   %doi={10.1063/1.523491},
}
\bib{Guevara}{article}{
   author={Guevara, Cristi Darley},
   title={Global behavior of finite energy solutions to the $d$-dimensional
   focusing nonlinear Schr\"{o}dinger equation},
   journal={Appl. Math. Res. Express. AMRX},
   date={2014},
   number={2},
   pages={177--243},
   issn={1687-1200},
  % review={\MR{3266698}},
  % doi={10.1002/cta.2381},
}

\bib{HR07}{article}{
   author={Holmer, Justin},
   author={Roudenko, Svetlana},
   title={On blow-up solutions to the 3D cubic nonlinear Schr\"{o}dinger
   equation},
   note={[Issue information previously given as no. 1 (2007)]},
   journal={Appl. Math. Res. Express. AMRX},
   date={2007},
   pages={Art. ID abm004, 31},
   issn={1687-1200},
   %review={\MR{2354447}},
}

\bib{HR2010}{article}{
   author={Holmer, {J.}},
   author={Roudenko, {S.}},
   title={Divergence of infinite-variance nonradial solutions to the 3D NLS
   equation},
   journal={Comm. Partial Differential Equations},
   volume={35},
   date={2010},
   number={5},
   pages={878--905},
   issn={0360-5302},
   %review={\MR{2753623}},
  % doi={10.1080/03605301003646713},
}
					
		
\bib{HR}{article}{
   author={Holmer, {J}.},
   author={Roudenko, {S.}},
   title={A sharp condition for scattering of the radial 3D cubic nonlinear
   Schr\"odinger equation},
   journal={Comm. Math. Phys.},
   volume={282},
   date={2008},
   number={2},
   pages={435--467},
   issn={0010-3616},
   %review={\MR{2421484}},
   %doi={10.1007/s00220-008-0529-y},
}

\bib{Kav}{article}{
   author={Kavian, O.},
   title={A remark on the blowing-up of solutions to the Cauchy problem for
   nonlinear Schr\"{o}dinger equations},
   journal={Trans. Amer. Math. Soc.},
   volume={299},
   date={1987},
   number={1},
   pages={193--203},
   issn={0002-9947},
   %review={\MR{869407}},
   %doi={10.2307/2000489},
}
	
\bib{KM1}{article}{
   author={Kenig, Carlos E.},
   author={Merle, Frank},
   title={Global well-posedness, scattering and blow-up for the
   energy-critical, focusing, nonlinear Schr\"odinger equation in the radial
   case},
   journal={Invent. Math.},
   volume={166},
   date={2006},
   number={3},
   pages={645--675},
   issn={0020-9910},
   %review={\MR{2257393}},
   %doi={10.1007/s00222-006-0011-4},
}
\bib{KM2}{article}{
   author={Kenig, {C. E.}},
   author={Merle, Frank},
   title={Global well-posedness, scattering and blow-up for the
   energy-critical focusing non-linear wave equation},
   journal={Acta Math.},
   volume={201},
   date={2008},
   number={2},
   pages={147--212},
   issn={0001-5962},
  % review={\MR{2461508}},
  % doi={10.1007/s11511-008-0031-6},
}

\bib{KRRT}{article}{
   author={Kuznetsov, E. A.},
   author={Rasmussen, J. Juul},
   author={Rypdal, K.},
   author={Turitsyn, S. K.},
   title={Sharper criteria for the wave collapse},
   note={The nonlinear Schr\"{o}dinger equation (Chernogolovka, 1994)},
   journal={Phys. D},
   volume={87},
   date={1995},
   number={1-4},
   pages={273--284},
   issn={0167-2789},
   %review={\MR{1361691}},
   %doi={10.1016/0167-2789(95)00150-3},
}

\bib{LMS}{article}{
   author={Lahaye, T.},
   author={Menotti, C.},
   author={Santos, L.},
   author={Lewenstein, M.},
   author={Pfau, T.},
   title={The physics of dipolar bosonic quantum gases},
   journal={Reports on Progress in Physics},
   volume={72}
   date={2009},
   number={12}
   pages={126401},
  % review={\MR{2461508}},
   %doi={10.1007/s11511-008-0031-6},
}

\bib{LW}{article}{
   author={Lu, J.},
   author={Wu, Y.},
   title={Sharp threshold for scattering of a generalized Davey-Stewartson system in three dimension},
   journal={Comm. Pure Appl. Anal.},
   date={2015},
   number={14},
   pages={ 1641--1670},
  % review={\MR{2461508}},
   %doi={10.1007/s11511-008-0031-6},
}
\bib{Mar}{article}{
   author={Martel, Yvan},
   title={Blow-up for the nonlinear Schr\"{o}dinger equation in nonisotropic
   spaces},
   journal={Nonlinear Anal.},
   volume={28},
   date={1997},
   number={12},
   pages={1903--1908},
   issn={0362-546X},
   %review={\MR{1436360}},
   %doi={10.1016/S0362-546X(96)00036-3},
}

\bib{NaPeSa}{article}{
   author={Nath, R.},
   author={Pedri, P.},
   author={Zoller, P.},
   author={Lewenstein, M.},
   title={Soliton-soliton scattering in dipolar Bose-Einstein condensates},
   journal={Phys. Rev. A},
   date={2007},
   number={76},
   pages={ 013606--013613},
  % review={\MR{2461508}},
   %doi={10.1007/s11511-008-0031-6},
}


\bib{OT}{article}{
   author={Ogawa, Takayoshi},
   author={Tsutsumi, Yoshio},
   title={Blow-up of $H^1$ solution for the nonlinear Schr\"odinger equation},
   journal={J. Differential Equations},
   volume={92},
   date={1991},
   number={2},
   pages={317--330},
   issn={0022-0396},
   %review={\MR{1120908}},
   %doi={10.1016/0022-0396(91)90052-B},
}
\bib{PS}{book}{
   author={Pitaevskii, Lev},
   author={Stringari, Sandro},
   title={Bose-Einstein condensation},
   series={International Series of Monographs on Physics},
   volume={116},
   publisher={The Clarendon Press, Oxford University Press, Oxford},
   date={2003},
   pages={x+382},
   isbn={0-19-850719-4},
   %review={\MR{2012737}},
}
\bib{SSZL}{article}{
   author={Santos, L.},
   author={Shlyapnikov, G.},
   author={Zoller, P.},
   author={Lewenstein, M.},
   title={Bose-Einstein condensation in trapped dipolar gases},
   journal={Phys. Rev. Lett.},
   date={2000},
   number={85},
   pages={ 1791--1797},
  % review={\MR{2461508}},
   %doi={10.1007/s11511-008-0031-6},
}
\bib{YY1}{article}{
   author={Yi, S.},
   author={You, L.},
   title={Trapped atomic condensates with anisotropic interactions},
   journal={Phys. Rev. A},
   volume={61}
   date={2000},
   number={4}
   pages={041604},
  % review={\MR{2461508}},
   %doi={10.1007/s11511-008-0031-6},
}
\bib{YY2}{article}{
   author={Yi, {S.}},
   author={You, {L.}},
   title={Trapped condensates of atoms with dipole interactions},
   journal={Phys. Rev. A},
   volume={63}
   date={2001},
   number={5}
   pages={053607},
  % review={\MR{2461508}},
   %doi={10.1007/s11511-008-0031-6},
}




\end{biblist}
\end{bibdiv}
\end{document}